\documentclass[11pt]{article}

\usepackage{amsmath,amsthm,amssymb,amsxtra,amsfonts,amsbsy,mathrsfs}
\usepackage[all,2cell]{xy} \UseAllTwocells
\SilentMatrices
\newtheorem{theorem}[subsection]{Theorem}
\newtheorem{subtheorem}[subsubsection]{Theorem}

\newtheorem{subcorollary}[subsubsection]{Corollary}
\newtheorem{lemma}[subsection]{Lemma}

\newtheorem{subproposition}[subsubsection]{Proposition}

\newtheorem{subdefinition}[subsubsection]{Definition}
\newtheorem{sublemma}[subsubsection]{Lemma}
\theoremstyle{definition}
\newtheorem{notation-convention}[subsection]{Notations
and Conventions}

\newtheorem{subremark}[subsubsection]{Remark}

\newtheorem{anitem}[subsubsection]{}
\newtheorem{blank}[subsection]{}

\newcommand{\leftexp}[2]{{\vphantom{#2}}^{#1}{#2}}

\newcommand{\bb}{\mathbb}
\newcommand{\s}{\mathscr}
\newcommand{\fk}{\mathfrak}

\begin{document}

\title{Generic Base Change, Artin's Comparison Theorem, and the
Decomposition Theorem for Complex Artin stacks}
\author{Shenghao Sun}
\date{}
\maketitle

\begin{abstract}
We prove the generic base change theorem for stacks, and give an
exposition on the lisse-analytic topos of complex analytic stacks,
proving some comparison theorems between various derived
categories of complex analytic stacks. This enables us to deduce the
decomposition theorem for perverse sheaves on complex Artin stacks
with affine stabilizers from the case over finite fields.
\end{abstract}

\section{Introduction}\label{sec-intro}

In the study of the topology of complex algebraic varieties, the
notion of intersection (co)homology and the decomposition theorem
have played an important role. They are also quite useful in some
other fields, such as representation theory. See \cite{Bry} for a
detailed introduction.

In \cite{BL} the notion of perverse sheaves was generalized to
spaces with group actions (the so-called \textit{equivariant
perverse sheaves}), and the decomposition theorem was proved in this
case. The notion of (middle) perverse sheaves has also been
generalized to algebraic stacks \cite{LO3}, and the decomposition
theorem has been proved for algebraic stacks of finite type with
affine stabilizers over a finite field \cite{Sun2}. The result for
algebraic stacks over the complex numbers was also announced in
(\cite{Sun2}, 3.15), and we publish the proof in this article. For
the necessity of the assumption on the stabilizers, we direct the
reader to (\cite{Sun2}, Section 1) for a counter-example of
Drinfeld.

\begin{blank}\label{affine-stab}
Let $k$ be a field and let $\mathcal X$ be a $k$-algebraic stack. We
say that $\mathcal X$ has \textit{affine stabilizers} if for every
$x\in\mathcal X(\overline{k}),$ the group scheme $\text{Aut}_x$ is
affine. Note that, since being affine is fpqc local on the base,
for any finite field extension $k'/k$ and any $x\in\mathcal X(k'),$ the $k'$-group scheme $\text{Aut}_x$ is affine.
\end{blank}

Here is the main result, in a simplified and global form.

\begin{theorem}
Let $f:X\to Y$ be a proper morphism of finite diagonal between
complex Artin stacks of finite type, with affine stabilizers, and
let $f^{\emph{an}}:X^{\emph{an}}\to Y^{\emph{an}}$ be the associated
morphism of complex analytic stacks. Then there exist locally closed
irreducible smooth substacks $Y_{\alpha}\subset Y,$ irreducible $\bb C$-local systems $L_{\alpha\beta}$ on $Y_{\alpha},$ and integers $d_{\alpha\beta}\ge0,$ the index set for $(\alpha,\beta)$ being finite, such that we have a decomposition
$$
IH^n(X^{\emph{an}},\bb
C)\simeq\bigoplus_{\alpha,\beta}IH^{n-d_{\alpha\beta}}
(\overline{Y}_{\alpha}^{\emph{an}},L_{\alpha\beta})
$$
for each $n\in\bb Z.$
\end{theorem}

See (\ref{6.2.5}) for the general and local version of the theorem.

We briefly mention some technical issues. One would like to deduce
the decomposition theorem over $\bb C$ from that over finite fields,
as in (\cite{BBD}, Section 6). In order for the argument to work,
one must generalize the generic base change theorem to stacks. Also,
in order to obtain a topological statement, one has to prove some
comparison theorems between different topologies. Roughly speaking,
the generic base change theorem relates lisse-\'etale sheaves over
$\bb C$ with lisse-\'etale sheaves over $\bb F$ (algebraic closure of a finite field), and the comparison
theorems relate lisse-\'etale sheaves over $\bb C$ with
lisse-analytic sheaves over $\bb C.$

\textbf{Organization.} In $\S\ref{sec-gbc}$ we prove the generic
base change for $Rf_*$ and $R\mathscr Hom,$ and in
$\S\ref{sec-analytic}$ we develop the theory of constructible
sheaves and their derived categories on complex analytic stacks that
are \textit{algebraic;} in particular, we give the comparison
between the lisse-\'etale topos and the lisse-analytic topos, and
the comparison between the adic version and the topological version
of the derived category of the lisse-analytic topos. In
$\S\ref{sec-decomp-C},$ after giving a comparison between bounded
derived categories with prescribed stratification over the complex
numbers and over an algebraic closure of a finite field, we finish
the proof of the decomposition theorem for stacks over $\bb C.$

\begin{notation-convention}\label{notat-conv}

\begin{anitem}\label{nc-adic}
Let
$(\Lambda,\mathfrak m)$ be a complete DVR of mixed characteristic,
with finite residue field $\Lambda_0$ of characteristic $\ell$ and
uniformizer $\lambda.$ Let $\Lambda_n=\Lambda/\mathfrak m^{n+1}$, for $n\in\bb N.$
\end{anitem}

\begin{anitem}\label{ft}
By an Artin stack, or an algebraic stack, we mean an algebraic stack
in the sense of M. Artin (\cite{Ols2}, 1.2.22) \textit{of finite
type} over the base.
We will use $\mathcal{X,Y,}\cdots$ to denote
algebraic stacks over a general base $S$ in $\S\ref{sec-gbc}.$ In
$\S\ref{sec-analytic}$ we use them to denote complex algebraic
stacks, and $\fk{X,Y,}\cdots$ for complex analytic stacks. By a presentation of an algebraic stack $\mathcal X,$ we mean a smooth surjection $\pi:X\to\mathcal X$ where $X$ is a scheme.
\end{anitem}

\begin{anitem}
By a variety
over $k$ we mean a separated reduced $k$-scheme of finite
type. For a $k$-algebraic stack $\mathcal X,$ we say that it is
\textit{essentially smooth} if $(\mathcal X_{\overline{k}})
_{\text{red}}$ is smooth over $\overline{k}.$
\end{anitem}

\begin{anitem}
For a map $f:X\to Y$ and a complex of sheaves $K$ on $Y,$ we
sometimes write $H^n(X,K)$ for $H^n(X,f^*K).$
\end{anitem}

\begin{anitem}\label{nota-derived}
We will denote $Rf_*,Rf_!,Lf^*$ and $Rf^!$ by
$f_*,f_!,f^*$ and $f^!$ respectively in most part of this paper,
except in (\ref{compar}, \ref{conts-derived}), where we use symbols like
$R\gamma_*$ and $R\pi_*$ to emphasize that we are considering
the derived functors.
\end{anitem}

\begin{anitem}
We will only consider the middle perversity. We use $\leftexp{p}{\s H}^i$
and $\leftexp{p}{\tau}_{\le i}$ to denote cohomology and truncations with
respect to this perverse $t$-structure.
\end{anitem}
\end{notation-convention}

\begin{flushleft}
\textbf{Acknowledgment.}
\end{flushleft}

I would like to thank my advisor Martin Olsson for introducing this
topic to me, and giving so many suggestions during the writing. Yves
Laszlo pointed out some mistakes and gave many helpful comments.
(\ref{abelian-rank}) is added to correct a mistake pointed out by
Ofer Gabber. Brian Conrad and Matthew Emerton have helped to answer my questions related to this paper on the MathOverflow website.
I thank the referee for valuable comments, which help to improve the article. The
revision of the paper was done during the stay in Ecole
polytechnique CMLS (UMR 7640) and Universit\'e Paris-Sud (UMR 8628),
while I was supported by ANR grant G-FIB.

\section{Generic base change}\label{sec-gbc}

As mentioned in $\S\ref{sec-intro},$ an important step in deducing
the decomposition theorem over $\bb C$ from that over $\bb F_q$ (as
in \cite{BBD}, Section 6) will be to compare the derived categories
of the fiber over $\bb C$ and the fiber over $\bb F$ of some stack
over a local ring with mixed characteristics. For doing that, we prove the
generic base change theorem (as in \cite{SGA4.5}, Th.\ finitude) for
stacks in this section.

\begin{blank}\label{adic-pair}
Let $S$ be a scheme satisfying the following condition denoted (LO):
it is a noetherian affine excellent finite-dimensional scheme in
which $\ell$ is invertible, and all $S$-schemes of finite type have
finite $\ell$-cohomological dimension. The theory of derived
categories and the six operations in \cite{LO1,LO2} then applies to
algebraic stacks over $S$ locally of finite type. As
mentioned in (\ref{ft}), we will only consider those
\textit{of finite type} over $S.$

We refer to (\cite{Sun}, $\S3$) for the definition and
basic properties of stratifiable complexes in detail; see also (\cite{LO2}, Section 3) or (\cite{Sun}, Definition 2.2) for more discussion on the $\lambda$-adic derived category of an algebraic stack. Here we only give a quick review of the definitions.

Let $\s A=\s A(\mathcal
X)$ be the abelian category $\text{Mod}(\mathcal X^{\bb
N}_{\text{lis-\'et}},\Lambda_{\bullet})$ of $\Lambda_{\bullet}$-modules on the simplicial lisse-\'etale topos of $\cal X$, and
let $\s{D(A)}$ be the ordinary derived category of $\s A.$ An object $M\in\s{D(A)}$ is called a \textit{$\lambda$-complex} (resp.\ an \textit{AR-null complex}) if all cohomology systems $\s H^i(M)$ are AR-adic (resp.\ AR-null). Let $\s D_c(\s A)$ be the full subcategory of $\lambda$-complexes, and let $D_c(\mathcal X,\Lambda),$ the \textit{adic derived category of $\mathcal X,$} to be the quotient of $\s D_c(\s A)$ by the full subcategory of AR-null complexes.

For a pair $(\s S,\mathcal L),$ where $\s S$
is a stratification of the algebraic stack $\mathcal X,$ and
$\mathcal L$ assigns to every stratum $\mathcal
U\in\s S$ a finite set $\mathcal{L(U)}$ of
isomorphism classes of simple locally constant constructible
(abbreviated as \textit{lcc}) $\Lambda_0$-sheaves on
$\mathcal U,$ we define $\s D_{\s S,\mathcal L}(\s A)$ to be the
full subcategory of $\s D_c(\s A)$ consisting of
complexes of projective systems $K=(K_n)_n$ such that, for all
$i,n\in\bb Z$ and for every $\mathcal U\in\s
S,$ the restrictions $\s H^i(K_n)|_{\mathcal U}$ are
lcc with Jordan-H\"older components contained in
$\mathcal{L(U)}.$ Define $D_{\s S,\mathcal L}(
\mathcal X,\Lambda)$ to be its essential image under the
localization $\s D_c(\s A)\to D_c(\mathcal X,\Lambda);$ in other words,
it is the quotient of $\s D_{\s S,\mathcal L}(\s A)$ by the
thick subcategory of AR-null complexes. It is a
triangulated category. Similarly, one can define, for each $n\ge0,$ a triangulated full subcategory $D_{\s S,\mathcal L}(\mathcal X,\Lambda_n)$ of $D_c(\mathcal X,\Lambda_n):$ it consists of those complexes $K$ such that $\s H^i(K)|_{\mathcal U}$ are lcc with Jordan-H\"older components contained in $\mathcal{L(U)},$ for each integer $i$ and stratum $\mathcal U\in\s S.$

Finally we define $D_c^{\text{stra}}(\mathcal X,\Lambda)$ to be the 2-direct limit of all the $D_{\s S,\mathcal L}(\mathcal X,\Lambda)$'s; similarly for $D_c^{\text{stra}}(\mathcal X,\Lambda_n).$
\end{blank}

\begin{blank}\label{pushforward-gbc}
For a morphism $f:\mathcal X\to\mathcal Y$ of
$S$-algebraic stacks and $K\in D_c^+(\mathcal
X,\Lambda_n)$ (resp.\ $D_c^+(\mathcal X,\Lambda)$), we say
that \textit{the formation of $f_*K$ commutes with generic
base change,} if there exists an open dense subscheme
$U\subset S$ such that for any morphism $g:S'\to U\subset
S$ with $S'$ satisfying (LO), the base change morphism
$g'^*f_*K\to f_{S'*}g''^*K$ is an isomorphism.
Recall that
the base change morphism is defined as follows: one applies
$f_*$ to the adjunction map $K\to g''_*g''^*K$, and then
uses the adjunction $(g'^*,g'_*)$ to obtain the base change morphism, as shown in the following 2-Cartesian diagram
$$
\xymatrix@C=.6cm{
\mathcal X \ar[d]_-f && \mathcal X_{S'} \ar[ll]_-{g''}
\ar[d]^-{f_{S'}} \\
\mathcal Y \ar[d] && \mathcal Y_{S'} \ar[ll]_-{g'} \ar[d]
\\
S & U \ar@{_{(}->}[l] & S'. \ar[l]_-g}
$$
\end{blank}

\begin{lemma}\label{L4.1}
\emph{(i)} Let $P:Y\to\mathcal Y$ be a presentation, and let the
following diagram be 2-Cartesian:
$$
\xymatrix@C=.8cm{
\mathcal X \ar[d]_-f & \mathcal X_Y \ar[l]_-{P'}
\ar[d]^-{f'} \\
\mathcal Y & Y. \ar[l]^-P}
$$
Then for $K\in D_c^+(\mathcal X,\Lambda)$, the formation of $f_*K$
commutes with generic base change if and only if the
formation of $f'_*(P'^*K)$ commutes with generic base
change.

\emph{(ii)} Let $K'\to K\to K''\to K'[1]$ be an exact triangle
in $D_c^+(\mathcal X,\Lambda)$, and let $f:\mathcal X\to\mathcal Y$ be an
$S$-morphism. If the formations of $f_*K'$ and $f_*K''$
commute with generic base change, then so does the formation
of $f_*K.$

\emph{(iii)} Let $f:\mathcal X\to\mathcal Y$ be a schematic
morphism, and let $K\in D^+_{\{\mathcal X\},\mathcal
L}(\mathcal X,\Lambda)$ for some
finite set $\mathcal L$ of isomorphism classes of simple lcc
$\Lambda_0$-sheaves on $\mathcal X.$ Then the formation
of $f_*K$ commutes with generic base change.

\emph{(iv)} Let $K\in D_c^+(\mathcal X,\Lambda)$, and let
$j:\mathcal U\to\mathcal X$ be an open immersion with
complement $i:\mathcal Z\to\mathcal X.$ For $g:S'\to S,$
consider the following diagram obtained by base change:
$$
\xymatrix@C=.9cm @R=.8cm{
&& \mathcal U_{S'} \ar[lld]_-{g_{\mathcal U}} \ar@{^{(}->}
[r]^-{j_{S'}} & \mathcal X_{S'} \ar[lld]^(.3){g''} \ar[d]
^(.7){f_{S'}} & \mathcal Z_{S'} \ar[lld]^(.3){g_{\mathcal
Z}} \ar@{_{(}->}[l]_-{i_{S'}} \\
\mathcal U \ar@{^{(}->}[r]_-j & \mathcal X \ar[d]^-f &
\mathcal Z \ar@{_{(}->}[l]^-i & \mathcal Y_{S'}
\ar[lld]^-{g'} & \\
& \mathcal Y &&&}.
$$
Suppose that the base change morphisms
\begin{gather*}
g'^*(fj)_*(j^*K)\longrightarrow(f_{S'}j_{S'})_*g_{\mathcal
U}^*(j^*K), \\
g'^*(fi)_*(i^!K)\longrightarrow(f_{S'}i_{S'})_*g_{\mathcal
Z}^*(i^!K) \qquad\emph{and} \\
g''^*j_*(j^*K)\longrightarrow j_{S'*}g_{\mathcal U}^*(j^*K)
\end{gather*}
are isomorphisms, then the base change morphism
$g'^*f_*K\to f_{S'*}g''^*K$ is also an isomorphism.

\emph{(v)} Let $f:\mathcal X\to\mathcal Y$ be a schematic
morphism of $S$-algebraic stacks, and let $K\in
D_c^{+,\emph{stra}}(\mathcal X,\Lambda)$. Then the
formation of $f_*K$ commutes with generic base change on
$S.$

\emph{(vi)} Let $f:\mathcal X\to\mathcal Y$ be a morphism of
$S$-algebraic stacks, and let $j:\mathcal U\to\mathcal Y$ be
an open immersion with complement $i:\mathcal Z\to\mathcal
Y.$ Let $K\in D_c^{+,\emph{stra}}(\mathcal X,\Lambda)$. Then there exists an open dense subscheme $S^0\subset S,$ such that for any
map $g:S'\to S,$ with associated diagram in which the
squares are 2-Cartesian:
$$
\xymatrix@C=.7cm @R=.6cm{
& \mathcal X_{\mathcal U,S'} \ar@{^{(}->}[rr]^-{j'_{S'}}
\ar'[d]^-{f_{\mathcal U_{S'}}}[dd]
\ar[ld]_-{g''_{\mathcal U}} && \mathcal X_{S'}
\ar'[d]^-{f_{S'}}[dd] \ar[ld]_-{g''} && \mathcal
X_{\mathcal Z,S'} \ar[dd]^-{f_{\mathcal Z_{S'}}}
\ar[ld]_-{g''_{\mathcal Z}} \ar@{_{(}->}[ll]_-{i'_{S'}} \\
\mathcal X_{\mathcal U} \ar@{^{(}->}[rr]_(.3){j'}
\ar[dd]_-{f_{\mathcal U}} && \mathcal X \ar[dd]^(.3)f &&
\mathcal X_{\mathcal Z} \ar@{_{(}->}[ll]^(.7){i'}
\ar[dd]_(.3){f_{\mathcal Z}} & \\
& \mathcal U_{S'} \ar[ld]_-{g'_{\mathcal U}}
\ar@{^{(}->}'[r]_-{j_{S'}}[rr] && \mathcal Y_{S'}
\ar[ld]^(.4){g'} && \mathcal Z_{S'},
\ar@{_{(}->}'[l][ll]^-{i_{S'}} \ar[ld]^-{g'_{\mathcal Z}}
\\
\mathcal U \ar@{^{(}->}[rr]_-j && \mathcal Y && \mathcal Z
\ar@{_{(}->}[ll]^-i}
$$
if the base change morphisms
$$
g_{\mathcal U}'^*f_{\mathcal U*}(j'^*K)\to f_{\mathcal
U_{S'}*}g''^*_{\mathcal U}(j'^*K)\quad\text{and}\quad
g'^*_{\mathcal Z}f_{\mathcal Z*}(i'^!K)\to f_{\mathcal
Z_{S'}*}g''^*_{\mathcal Z}(i'^!K)
$$
are isomorphisms, then the base change morphism $g'^*f_*K\to f_{S'*}g''^*K$ is an isomorphism over $\mathcal Y_{S'\times_SS^0}$.
\end{lemma}

Similar results hold with $\Lambda$ replaced by $\Lambda_n\ (n\ge0),$ and the proof is the same.

\begin{proof}
(i) Given a map $g:S'\to S,$ consider the following
diagram
$$
\xymatrix@C=.6cm @R=.6cm{
& \mathcal X_Y \ar[ld]_-{P'} \ar'[d][dd]_-{f'} && \mathcal
X_{Y,S'} \ar[ll]_-{g_Y''} \ar[ld]^-{P'_{S'}}
\ar[dd]^-{f'_{S'}} \\
\mathcal X \ar[dd]_-f && \mathcal X_{S'} \ar[ll]_(.3){g''}
\ar[dd]^(.3){f_{S'}} & \\
& Y \ar[ld]^-P && Y_{S'} \ar'[l][ll]^(.4){g_Y'}
\ar[ld]^-{P_{S'}} \\
\mathcal Y && \mathcal Y_{S'} \ar[ll]_-{g'} &}
$$
where all squares are 2-Cartesian. For the base change
morphism $g'^*f_*K\to f_{S'*}g''^*K$ to be an isomorphism
on $\mathcal Y_{S'},$ it suffices for it to be an
isomorphism locally on $Y_{S'}.$ In the following
commutative diagram
$$
\xymatrix@C=1cm{
P_{S'}^*g'^*f_*K \ar[r]_-{(0)} \ar@{=}[d]_-{(1)} &
P_{S'}^*f_{S'*}g''^*K \ar[d]^-{(2)} \\
g'^*_YP^*f_*K \ar[d]_-{(3)} & f'_{S'*}P'^*_{S'}g''^*K
\ar@{=}[d]^-{(4)} \\
g'^*_Yf'_*P'^*K \ar[r]^-{(5)} & f'_{S'*}g''^*_YP'^*K,}
$$
(1) and (4) are canonical isomorphisms given by
$``P^*g^*\simeq g^*P^*",$ and (2) and (3) are canonical
isomorphisms given by $``P^*f_*=f_*P^*",$ which follows
from the definition of $f_*$ on the lisse-\'etale site.
Therefore, (0) is an isomorphism if and only if (5) is
an isomorphism.

(ii) This follows easily from the axioms of a triangulated
category (or 5-lemma):
$$
\xymatrix@C=.8cm{
g'^*f_*K' \ar[r] \ar[d]^-{\sim} & g'^*f_*K \ar[r] \ar[d] &
g'^*f_*K'' \ar[r] \ar[d]^-{\sim} & \\
f_{S'*}g''^*K' \ar[r] & f_{S'*}g''^*K \ar[r] &
f_{S'*}g''^*K'' \ar[r] &.}
$$

(iii) By (i) we may assume that $f:X\to Y$ is a morphism
of $S$-schemes. Note that the property of being
trivialized by a pair of the form $(\{\mathcal
X\},\mathcal L)$ is preserved when passing to a
presentation. By definition $f_*K$ is the class of the
system $(f_*\widehat{K}_n)_n,$ so it suffices to show that
there exists an open dense subscheme of $S$ over which
the formation of $f_*\widehat{K}_n$ commutes with base
change, for every $n.$ By the spectral sequence
$$
R^pf_*\mathscr H^q(\widehat{K}_n)\Longrightarrow
R^{p+q}f_*\widehat{K}_n
$$
and (ii), it suffices to show the existence of an
open dense subscheme of $S,$ over which the formations of $f_*L$
commute with base change, for all $L\in\mathcal
L.$ This follows from (\cite{SGA4.5}, Th. finitude).

(iv) Consider the commutative diagram
$$
\xymatrix@C=.5cm @R=.7cm{
g'^*f_*i_*i^!K \ar@{=}[d]_-{(1)} \ar[r]_-{(2)} &
f_{S'*}g''^*i_*i^!K \ar[r]_-{(3)} &
f_{S'*}i_{S'*}g_{\mathcal Z}^*i^!K \ar@{=}[d]^-{(4)} \\
g'^*(fi)_*i^!K \ar[rr]^-{(5)} &&
(f_{S'}i_{S'})_*g_{\mathcal Z}^*i^!K.}
$$
(1) and (4) are canonical isomorphisms, (5) is an
isomorphism by assumption, and (3) is the base change
morphism for $i_*,$ which is an isomorphism by
(\cite{LO2}, 12.5.3), since $i_*=i_!.$ Therefore, (2) is
an isomorphism. Similarly, consider the commutative
diagram
$$
\xymatrix@C=.5cm @R=.7cm{
g'^*f_*j_*j^*K \ar@{=}[d]_-{(1)} \ar[r]_-{(2)} &
f_{S'*}g''^*j_*j^*K \ar[r]_-{(3)} & f_{S'*}j_{S'*}
g_{\mathcal U}^*j^*K \ar@{=}[d]^-{(4)} \\
g'^*(fj)_*j^*K \ar[rr]^-{(5)} && (f_{S'}j_{S'})_*
g_{\mathcal U}^*j^*K.}
$$
(1) and (4) are canonical isomorphisms, and (3) and (5)
are isomorphisms by assumption, so (2) is an isomorphism.
Then apply (ii) to the exact triangle $i_*i^!K\to K\to
j_*j^*K\to.$

(v) By (i), we may assume that $f:X\to Y$ is a morphism of
$S$-schemes. Assume that $K$ is trivialized by $(\mathscr
S,\mathcal L),$ and let $j:U\to X$ be the immersion of an
open stratum in $\mathscr S$ with complement $i:Z\to X.$
Then
$j^*K\in D^+_{\{U\},\mathcal L(U)}(U,\Lambda)$, so by
(iii), the formation of $j_*(K|_U)$ commutes with generic
base change. This is the third base change isomorphism in
the assumption of (iv). By noetherian induction and (iv),
we may replace $X$ by $U$ and assume that $\mathscr S=\{X\}.$
The result follows from (iii).

(vi) In the commutative diagrams
$$
\xymatrix@C=.7cm @R=.5cm{
g'^*j_*f_{\mathcal U*}j'^*K \ar[rr]^-{(1)} \ar[d]_-{(2)}
&& j_{S'*}g'^*_{\mathcal U}f_{\mathcal U*}j'^*K
\ar[d]^-{(3)} \\
g'^*(fj')_*j'^*K \ar[r]_-{(4)} &
(f_{S'}j'_{S'})_*g''^*_{\mathcal U}j'^*K \ar[r]_-{(5)} &
j_{S'*}f_{\mathcal U_{S'}*}g''^*_{\mathcal U}j'^*K}
$$
and
$$
\xymatrix@C=.7cm @R=.5cm{
g'^*i_*f_{\mathcal Z*}i'^!K \ar[rr]^-{(6)} \ar[d]_-{(7)}
&& i_{S'*}g'^*_{\mathcal Z}f_{\mathcal Z*}i'^!K
\ar[d]^-{(8)} \\
g'^*(fi')_*i'^!K \ar[r]_-{(9)} & (f_{S'}i'_{S'})_*
g''^*_{\mathcal Z}i'^!K \ar[r]_-{(10)} &
i_{S'*}f_{\mathcal Z_{S'}*}g''^*_{\mathcal Z}i'^!K,}
$$
(2), (5), (7) and (10) are canonical isomorphisms, (3) and
(8) are isomorphisms by assumption, (6) is an isomorphism
by proper base change, and (1) is an isomorphism after
shrinking $S$ by (v). Therefore, (4) and (9) are
isomorphisms. Also by (v), the base change morphism
$g''^*j'_*(j'^*K)\to j'_{S'*}g''^*_{\mathcal U}(j'^*K)$
becomes an isomorphism after shrinking $S.$ Hence by (iv),
the base change morphism $g'^*f_*K\to f_{S'*}g''^*K$ is
an isomorphism after shrinking $S.$
\end{proof}

\begin{blank}\label{RHom-gbc}
For $K\in D_c^-(\mathcal X,\Lambda_n)$ and $L\in
D_c^+(\mathcal X,\Lambda_n),$ and for a morphism
$g:\mathcal Y\to\mathcal X,$ the base change morphism
$g^*R\s Hom_{\mathcal X}(K,L)\to R\s
Hom_{\mathcal Y}(g^*K,g^*L)$ is defined as follows. By
adjunction $(g^*,g_*),$ it corresponds to the morphism
$$
R\s Hom_{\mathcal X}(K,L)\to g_*R\s Hom_{\mathcal Y}(g^*K,g^*L)\simeq
R\s Hom_{\mathcal X}(K,g_*g^*L)
$$
obtained by applying $R\s Hom_{\mathcal X}(K,-)$ to the
adjunction morphism $L\to g_*g^*L.$ One can define the base
change morphism for $\Lambda$-coefficients in the same way.

Note that if $K'\to K\to K''\to K'[1]$ is an exact triangle, and the
base change morphisms for $R\s Hom(K',L)$ and $R\s Hom(K'',L)$ are
isomorphisms, then so is the base change morphism for $R\s Hom(K,L)$; similarly for the position of $L.$

We say that \textit{the formation of $R\s
Hom_{\mathcal X}(K,L)$ commutes with generic base change
on $S,$} if there exists an open dense subscheme $U\subset
S$ such that for any morphism $g:S'\to U\subset S$ with
$S'$ satisfying (LO), the base change morphism
$$
g'^*R\s Hom_{\mathcal X}(K,L)\to R\s Hom_{\mathcal X_{S'}}(g'^*K,g'^*L)
$$
is an isomorphism. Here $g':\mathcal X_{S'}\to\mathcal
X$ is the natural projection.
\end{blank}

The following is the main result of this section.

\begin{theorem}\label{Tbc}
\emph{(i)} Let $f:\mathcal X\to\mathcal Y$ be a morphism of
$S$-algebraic stacks. For every $K\in
D_c^{+,\emph{stra}}(\mathcal X,\Lambda_n)$ \emph{(}resp.\
$D_c^{+,\emph{stra}}(\mathcal X,\Lambda))$, the
formation of $f_*K$ commutes with generic base change on
$S.$

\emph{(ii)} For all $K,L\in D_c^b(\mathcal X,\Lambda_n),$ the formation of
$R\s Hom_{\mathcal X}(K,L)$ commutes with generic base change on $S.$
\end{theorem}

\begin{proof}
(i) We can always replace a stack by its maximal reduced
closed substack, so we will assume that all stacks in the proof
are reduced.

Suppose that $K$ is $(\s S,\mathcal L)$-stratifiable for
some pair $(\s S,\mathcal L).$ By (\ref{L4.1} i, iii, iv), we can replace
$\mathcal Y$ by a presentation and replace $\mathcal X$ by an open
stratum in $\s S,$ to assume that $\mathcal Y=Y$ is a
scheme, that $\s S=\{\mathcal X\},$ that the
relative inertia $\mathcal I_f$ is flat over $\mathcal X$
and has components over $\mathcal X$ (\cite{Beh2}, 5.1.14); let
$$
\xymatrix@C=.7cm{
\mathcal X \ar[r]^-{\pi} & X \ar[r]^-b & Y}
$$
be the rigidification with respect to $\mathcal I_f.$
Replacing $\mathcal X$ by the inverse image of an open
dense subscheme of the $S$-algebraic space $X,$ we may
assume that $X$ is a scheme. Let $\s F=\pi_*K,$ which is
stratifiable (\cite{Sun}, 3.9). By (\ref{L4.1} v), the
formation of $b_*\s F$ commutes with generic base
change. To finish the proof, we shall show that the
formation of $\pi_*K$ commutes with generic base change.
As in the proof of (\ref{L4.1} iii), it suffices to show
that there exists an open dense subscheme $U$ of $S$, over
which the formations of $\pi_*L$ commute with any base
change $g:S'\to U,$ for all $L\in\mathcal L.$

By (\cite{Beh2}, 5.1.5), $\pi$ is smooth, so \'etale
locally it has a section. By (\ref{L4.1} i) we may assume
that $\pi:BG\to X$ is a neutral gerbe, associated to a
flat group space $G/X.$ By (\ref{L4.1} vi) we can use
d\'evissage and shrink $X$ to an open subscheme.
Using the same technique as the proof of (\cite{Sun},
3.9), we can reduce to the case where $G/X$ is smooth. For
the reader's convenience, we briefly recall this
reduction. Shrinking $X,$ we may assume that $X$ is an integral
scheme with function field $k(X),$ and that $G/X$ is a group
scheme. There exists a finite field extension $k''/k(X)$
such that $G_{\text{red}}$ is smooth over $\text{Spec }k''.$
Let $k'$ be the separable closure of $k(X)$ in $k''.$
Purely inseparable
morphisms are universal homeomorphisms. By taking the
normalization of $X$ in these field extensions, we get a
finite generically \'etale surjection $X'\to X,$ such that
$G_{\text{red}}$ is generically smooth over $X'.$
Shrinking $X$ and $X'$ we may assume that $X'\to X$ is an
\'etale surjection, and replacing $X$ by $X'$
(\ref{L4.1} i) we may assume that $G_{\text{red}}$ is
generically smooth over $X,$ and shrinking $X$ further we may
assume that $G_{\text{red}}$ is smooth over $X.$ Finally we may  replace $G$ by $G_{\text{red}}$, since the morphism
$BG_{\text{red}}\to BG$ is representable and radicial.

Now $P:X\to BG$ is a presentation, and we consider the associated
smooth hypercover. Let $f_p:G^p\to X\ (p\ge1)$ be the
structural maps, and let the following squares be 2-Cartesian:
$$
\xymatrix@C=.8cm @R=.6cm{
G_{S'}^p \ar[r]^-{f_{p,S'}} \ar[d]_-{g_p} & X_{S'}
\ar[r]^-{P_{S'}} \ar[d]_-{g'} & (BG)_{S'} \ar[r]^-{\pi_{S'}}
\ar[d]_-{g''} & X_{S'} \ar[r] \ar[d]_-{g'} & S' \ar[d]^-g \\
G^p \ar[r]_-{f_p} & X \ar[r]_-P & BG \ar[r]_-{\pi} & X
\ar[r] & S.}
$$
We have the spectral sequence
(\cite{LO2}, 10.0.9)
$$
R^qf_{p*}f_p^*P^*L\Longrightarrow R^{p+q}\pi_*L
$$
and similarly for the base change to $S'.$ We can regard the
map $f_p$ as a product $\prod_pf_1$ and apply the K\"unneth
formula (shrinking $X$ we can assume that $X$
satisfies the condition (LO), and we can apply
(\cite{LO2}, 11.0.14))
$$
f_{p*}f_p^*P^*L=f_{1*}(f_1^*P^*L)\otimes_{\Lambda_0}f_{1*}\Lambda_0
\otimes_{\Lambda_0}\cdots\otimes_{\Lambda_0}f_{1*}\Lambda_0.
$$
Shrink $S$ so that the formations of $f_{1*}f_1^*P^*L$ and
$f_{1*}\Lambda_0$ commute with any base change on $S.$
From the base change morphism of the spectral sequences
$$
\xymatrix@C=2cm @R=.5cm{
g'^*R^qf_{p*}f_p^*P^*L \ar@{=>}[r] \ar@{=}[d] &
g'^*R^{p+q}\pi_*L \ar[ddddd]^-{(1)} \\
\mathscr H^qg'^*(f_{1*}f_1^*P^*L\otimes f_{1*}\Lambda_0\otimes
\cdots\otimes f_{1*}\Lambda_0) \ar@{=}[d] & \\
\mathscr H^q(g'^*f_{1*}f_1^*P^*L\otimes
g'^*f_{1*}\Lambda_0\otimes\cdots\otimes
g'^*f_{1*}\Lambda_0) \ar[d]^-{\sim} & \\
\mathscr H^q(f_{1,S'*}g_1^*f_1^*P^*L\otimes
f_{1,S'*}\Lambda_0\otimes\cdots\otimes
f_{1,S'*}\Lambda_0) \ar@{=}[d] & \\
\mathscr H^q(f_{1,S'*}f_{1,S'}^*P_{S'}^*g''^*L\otimes
f_{1,S'*}\Lambda_0\otimes\cdots\otimes
f_{1,S'*}\Lambda_0) \ar@{=}[d] & \\
R^qf_{p,S'*}f_{p,S'}^*P_{S'}^*g''^*L \ar@{=>}[r] &
R^{p+q}\pi_{S'*}g''^*L}
$$
we see that the base change morphism (1) is an
isomorphism.

(ii) Let $g:S'\to S$ be any morphism, $P:X\to\mathcal X$ be a
presentation, and consider the 2-Cartesian diagrams
$$
\xymatrix@C=1.2cm @R=.6cm{
X_{S'} \ar[r]^-{P'} \ar[d]_-{g''} & \mathcal X_{S'}
\ar[r] \ar[d]_-{g'} & S' \ar[d]^-g \\
X \ar[r]_-P & \mathcal X \ar[r] & S.}
$$
For the base change morphism
$$
g'^*R\s Hom_{\mathcal X}(K,L)\to R\s Hom_{\mathcal X_{S'}}(g'^*K,g'^*L)
$$
to be an isomorphism, we can check it locally on $X_{S'}.$
Consider the commutative diagram
$$
\xymatrix@C=1.2cm @R=.6cm{
P'^*g'^*R\s Hom_{\mathcal X}(K,L) \ar[r]^-{(1)}
\ar[d]_-{(2)} & P'^*R\s Hom_{\mathcal X_{S'}}(g'^*K,g'^*L)
\ar[d]^-{(3)} \\
g''^*P^*R\s Hom_{\mathcal X}(K,L) \ar[d]_-{(4)} &
R\s Hom_{X_{S'}}(P'^*g'^*K,P'^*g'^*L) \ar[d]^-{(5)} \\
g''^*R\s Hom_X(P^*K,P^*L) \ar[r]^-{(6)} &
R\s Hom_{X_{S'}}(g''^*P^*K,g''^*P^*L),}
$$
where (2) and (5) are canonical isomorphisms, (3) and (4)
are isomorphisms by (\cite{LO1}, 4.2.3), and (6) is an
isomorphism after shrinking $S$ (\cite{SGA4.5}, Th.\
finitude, 2.10). Therefore (1) is an isomorphism after
shrinking $S.$
\end{proof}

\begin{subremark}\label{}
(i) This result strengthens (\cite{Ols3}, 9.10 ii), in that
the open subscheme in $S$ can be chosen to be independent
of the index $i$ as in $R^if_*F.$

(ii) As we only used $f_*,$ not $f_!,$ in the proof of the generic base change theorem, it may seem that the hypothesis (LO) on the base $S$ (cf.\ \ref{adic-pair}) is unnecessary. However, in the proof of (\cite{Sun}, 3.9), when proving that $f_*$ preserves stratifiability, which is needed in (\ref{Tbc}), we worked with the case for $f_!$ first, in order to do noetherian induction. Possibly this hypothesis on cohomological dimension can be removed in the future.
\end{subremark}

\section{Complex analytic stacks}\label{sec-analytic}

In this section, we give some fundamental results on
constructible sheaves and derived categories on the
lisse-analytic topos of the analytification of a complex
algebraic stack. The two main results in this section are:
the comparison between the adic derived categories of the
lisse-\'etale topos and the lisse-analytic topos (\ref{P-compar}),
and the comparison between the adic derived category and
the topological derived category of the lisse-analytic topos
(\ref{agree}).

\subsection{Lisse-analytic topos}

Stacks over topological categories have already been discussed, for
instance in \cite{Noo, Toe}. Strictly speaking, To\"en only
discussed analytic Deligne-Mumford stacks in \cite{Toe}, and Noohi
only discussed topological stacks in \cite{Noo} (and mentioned
analytic stacks briefly).

Since we are mainly interested in analytifications of
complex algebraic stacks, and will not study analytic
spaces and analytic stacks in full generality in this paper,
we will make a global assumption on analytic spaces:
we only consider analytic spaces of \textit{finite
dimension.} This rules out infinite disjoint unions of spaces of increasing dimensions, and is consistent with our assumption that
algebraic stacks are of finite type (\ref{ft}).

A morphism $f:X\to Y$ of complex analytic spaces is
\textit{smooth} if for every point $x\in X,$ there exist
open neighborhoods $x\in U\subset X$ and $f(x)\in V\subset Y,$
with $f(U)=V,$ such that $f|_U:U\to V$ is isomorphic to the
projection $\text{pr}_1:V\times Z\to V$ for some complex manifold
$Z$ (one can certainly take $Z$ to be a polydisk). In topology,
this is usually called a \textit{submersion,} but we will
use the algebro-geometric terminology of \textit{smoothness}
in the paper, if there is no confusion.

\begin{subdefinition}\label{analytic-stack}
Let $\textbf{\emph{Ana-Sp}}$ be the site of complex
analytic spaces with the analytic topology. A stack
$\fk X$ over this site is called an \emph{analytic
stack}, if the following hold:

\emph{(i)} the diagonal $\Delta:\fk X\to\fk X\times\fk X$ is representable
(by analytic spaces) and, letting the inertia $\fk{I_X}$ of
$\fk X$ be the fiber product $\fk X\times_{\Delta,\fk X\times\fk X,\Delta}\fk
X$ with $p_1:\fk{I_X}\to\fk X$ the first projection, the
complex Lie group $p_1^{-1}(x)$ has finitely many connected
components, for every $x\in\fk X(\bb C),$ and

\emph{(ii)} there exists a smooth surjection $P:X\to
\fk X,$ where $X$ is an analytic space.
\end{subdefinition}

We will call $P:X\to\fk X$ in (ii) an \textit{analytic
presentation} of $\fk X.$

\begin{anitem}\label{lis-an}
Similar to the lisse-\'etale topos of an algebraic stack,
one can define the \textit{lisse-analytic topos}
$\fk X_{\text{lis-an}}$ of an analytic stack
$\fk X$ to be the topos associated to the
\textit{lisse-analytic site} $\text{Lis-an}(\fk X)$
defined as follows:

$\bullet$ Objects: pairs $(U,u:U\to\fk X),$ where
$U$ is an complex analytic space and $u$ is a smooth
morphism;

$\bullet$ Morphisms: a morphism $(U,u\in\fk X
(U))\to(V,v\in\fk X(V))$ is given by a pair
$(f,\alpha),$ where $f:U\to V$ is a morphism of analytic
spaces and $\alpha:u\cong vf$ is an isomorphism in
$\fk X(U);$ the composition law is evident;

$\bullet$ Open coverings: $\{(j_i,\alpha_i):
(U_i,u_i\in\fk X(U_i))\to(U,u\in\fk X(U))\}_{i\in I}$
is an open covering if the maps $j_i:U_i\to U$
are open immersions and their
images cover $U.$

As in (\cite{LMB}, 12.2.1), one can show that, to give a sheaf
$F\in\fk X_{\text{lis-an}}$ is equivalent to giving a sheaf
$F_{U,u}$ in the analytic topos $U_{\text{an}}$ of $U$ for every
$(U,u)\in\text{Lis-an}(\fk X),$ and a morphism
$\theta_{f,\alpha}:f^{-1}F_{V,v}\to F_{U,u}$ for every morphism
$(f,\alpha):(U,u)\to(V,v)$ in $\text{Lis-an} (\fk X),$ such that

$\bullet\ \theta_{f,\alpha}$ is an isomorphism if $f$ is
an open immersion, and

$\bullet$ for every composition
$$
\xymatrix@C=1.2cm{
(U,u) \ar[r]^-{(f,\alpha)} & (V,v) \ar[r]^-{(g,\beta)} &
(W,w)}
$$
we have $\theta_{f,\alpha}\circ f^{-1}(\theta_{g,\beta})
=\theta_{gf,\beta(f)\circ\alpha}.$

The sheaf $F$ is \textit{Cartesian} if $\theta_{f,\alpha}$ is an
isomorphism, for every $(f,\alpha).$ By abuse of notation, we will
also denote ``$F_{U,u}$" and ``$\theta_{f, \alpha}$" by ``$F_U$" and
``$\theta_f$" respectively, if there is no confusion about the
reference to $u$ and $\alpha.$

This topos is equivalent to the ``lisse-\'etale" topos
$\fk X_{\text{lis-\'et}}$ associated to the site
$\text{Lis-\'et}(\fk X)$ with the same underlying
category as that of $\text{Lis-an}(\fk X),$ but the
open coverings are surjective families of local
isomorphisms. This is because the two topologies are
cofinal: for a local isomorphism $V\to U$ of analytic
spaces, there exists an open covering $\{V_i\subset
V\}_i$ of $V$ by analytic subspaces, such that for each
$i,$ the composition $V_i\subset V\to U$ is an open immersion.
\end{anitem}

\begin{anitem}\label{cart-complex}
Let $C^{\bullet}$ be a complex of sheaves of abelian groups
in $\fk X_{\text{lis-an}}.$ For a morphism $f:U\to V$ in
$\text{Lis-an}(\fk X),$ we have $\theta^n_f:f^*C^n_V\to
C^n_U$ for each component $C^n,$ and these maps commute
with the differentials in $C^{\bullet}$ (by definition
of morphisms of sheaves), hence they give a chain map
$\theta_f^{\bullet}:f^*C^{\bullet}_V\to C^{\bullet}_U.$
If the cohomology sheaves $\s H^n(C^{\bullet})$ are all
Cartesian, then $\theta_f^{\bullet}$ is an quasi-isomorphism,
for every $f.$
\end{anitem}

\subsection{Locally constant sheaves and constructible
sheaves}

Let $R$ be a commutative ring with identity.
For a sheaf of sets (resp.\ a sheaf of $R$-modules) on
the analytic site of an analytic space,
we say that the sheaf is \textit{locally constant
constructible}, abbreviated as \textit{lcc}, if it is
locally constant with respect to the analytic topology,
and stalks are finite sets (resp.\ finitely generated
$R$-modules).

Let $\fk X$ be an analytic stack. For a Cartesian
sheaf $F\in\fk X_{\text{lis-an}},$ we say that $F$
is \textit{locally constant} (resp.\ \textit{lcc}) if the
conditions in the following (\ref{analytic-9.1}) hold. The following lemma is an analytic version of (\cite{Ols3}, 9.1).

\begin{sublemma}\label{analytic-9.1}
Let $F\in\fk X_{\emph{lis-an}}$ be a Cartesian
sheaf. Then the following are equivalent.

\emph{(i)} For every $(U,u)\in\emph{Lis-an}(\fk X),$ the
sheaf $F_U$ is locally constant (resp.\ lcc).

\emph{(ii)} There exists an analytic presentation $P:X\to
\fk X$ such that $F_X$ is locally constant (resp.\ lcc).

The same statement holds for a Cartesian sheaf $F$ of
$R$-modules.
\end{sublemma}

\begin{proof}
We only need to show that (ii) implies (i). There exists an open
covering $U=\cup U_i,$ such that over each $U_i,$ the smooth
surjection $X\times_{P,\fk X,u}U\to U$ has a section $s_i:$
$$
\xymatrix@C=.6cm @R=.8cm{
& X\times_{\fk X}U \ar[r] \ar[d] & X \ar[d]^-P \\
U_i \ar@{^{(}->}[r] \ar@{-->}[ur]^-{s_i} & U
\ar[r]^-u & \fk X.}
$$
Therefore $F_{U_i}\simeq s_i^{-1}F_{X\times_{\fk
X}U},$ which is locally constant (resp.\ lcc).
\end{proof}

\begin{anitem}\label{analytification}
Let $\mathcal X$ be a complex algebraic stack. Following
(\cite{Noo}, 20), one can define its \textit{associated analytic
stack} $\mathcal X^{\text{an}}$ as follows. If
$X_1\rightrightarrows X_0\to\mathcal X$ is a smooth
groupoid presentation, then $\mathcal X^{\text{an}}$ is
defined to be the analytic stack given by the
presentation $X_1^{\text{an}}\rightrightarrows
X_0^{\text{an}},$ and it can be proved that this is
independent of the choice of the presentation, up to an
isomorphism that is unique up to 2-isomorphism.
Similarly, for a morphism $f:\mathcal X\to\mathcal Y$ of
complex algebraic stacks, one can choose their
presentations so that $f$ lifts to a morphism of
groupoids, hence induces a morphism of their
analytifications, denoted $f^{\text{an}}:\mathcal
X^{\text{an}}\to\mathcal Y^{\text{an}}.$ The
analytification functor preserves finite 2-fiber products.

Sometimes we write $\mathcal X(\bb C)$ for the analytification
$\mathcal X^{\text{an}}$ or the associated lisse-analytic topos.
For a $\bb C$-algebraic space $X,$ we denote by $X(\bb C)$ the
analytification or the associated analytic topos. There is a
possible confusion which will not occur in the sequel:
for an analytic space $X,$ these two topoi are not the same.
The restriction functor defines an equivalence from
Cartesian sheaves $X_{\text{lis-an, cart}}$ to $X_{\text{an}}.$
\end{anitem}

\begin{anitem}\label{alg-constr}
Let $\fk X=\mathcal X^{\text{an}}$ for a complex
algebraic stack $\mathcal X,$ and let $P:X\to\mathcal X$
be a presentation. Let $R$ be a commutative ring with identity.
For a sheaf $F$ of sets (resp.\ $R$-modules) on $\fk
X_{\text{lis-an}},$ we say that $F$ is \textit{algebraically
constructible} (or just \textit{constructible}), if it is
Cartesian, and that for every $(U,u)\in\text{Lis-\'et}(\mathcal X),$
the sheaf $F_{U(\bb C)}$ is constructible, i.e.\ lcc on each
stratum in an \textit{algebraic} stratification of the
analytic space $U(\bb C).$ In the following, when
there seems to be a confusion about the coefficient ring
$R,$ we will mention it explicitly.

One could also define a notion of \textit{analytic
constructibility}, using analytic stratifications rather
than algebraic ones, but this notion will not give us a
comparison between the constructible derived categories
of the lisse-\'etale topos and of the lisse-analytic
topos.

The notion of constructible sheaves (and some variants) on
complex analytic spaces are defined in (\cite{Dim}, 4.1).
\end{anitem}

\begin{sublemma}\label{L-analytic-constr}
Let $F$ be a Cartesian sheaf of sets (resp.\ $R$-modules)
on $\fk X_{\emph{lis-an}}$, and let $P:X\to\mathcal X$ be a presentation as above. Then the following are equivalent.

\emph{(i)} $F$ is constructible.

\emph{(ii)} $F_{X(\bb C)}$ is constructible on
$X(\bb C)$ (in the algebraic sense above).

\emph{(iii)} There exists an algebraic stratification $\s S$ on $\fk
X,$ such that for each stratum $V(\bb C),$ the sheaf $F_{V(\bb
C)}$ is lcc.
\end{sublemma}

\begin{proof}
(i)$\Rightarrow$(ii) is clear.

(ii)$\Rightarrow$(iii). Let $\s S_X$ be a
stratification of the scheme $X,$ such that for each
$U\in\s S_X,$ the sheaf $F_{U(\bb C)}$ is lcc.
Let $U$ be an open stratum in $\s S_X,$ and let $V$ be the image of
$U$ under the map $P;$ then $V$ is an open substack of
$\mathcal X,$ and $P_U:U\to V$ is a presentation.
By (\ref{analytic-9.1}) we see that $F_{V^{\text{an}}}$
is lcc. Since $X-P^{-1}(V)\to\mathcal X-V$ gives an algebraic
presentation of $(\mathcal X-V)^{\text{an}}=\fk
X-V^{\text{an}},$ and
$$
(F|_{\fk X-V^{\text{an}}})_{(X-P^{-1}(V))^{\text
{an}}}\simeq F_{X^{\text{an}}}|_{(X-P^{-1}(V))^{\text{an}}}
$$
is still constructible, by noetherian induction we are done.

(iii)$\Rightarrow$(i). Let $(U,u)\in\text{Lis-\'et}(
\mathcal X).$ Then $u^{\text{an},*}\s S$
is an algebraic stratification of $U(\bb C),$ and it
is clear that $F_{U(\bb C)}$ is lcc on each stratum of
this stratification.
\end{proof}

\begin{anitem}\label{constr=serre}
Assume the ring $R$ is noetherian. Then the
constructible $R$-modules on $\fk X_{\text{lis-an}}$ form a
full subcategory $\text{Mod}_c(\fk X,R)$ of
$\text{Mod}(\fk X,R)$ that is closed under
taking kernels, cokernels and extensions (i.e.\
it is a Serre subcategory). To see this, we first
show that Cartesian sheaves form a Serre subcategory.

Let $(f,\alpha):(U,u)\to(V,v)$ be a morphism in
$\text{Lis-an}(\fk X).$ The functor
$f^*:\text{Mod}(V_{\text{an}},R)\to
\text{Mod}(U_{\text{an}},R)$ is exact, because
$f^*F=R_U\otimes_{f^{-1}R_V}f^{-1}F
=f^{-1}F.$ Let $a:F\to G$ be a morphism of Cartesian
sheaves. Then $\text{Ker}(f^*a_V:f^*F_V\to f^*G_V)=f^*
\text{Ker}(a_V),$ and it is clear that the induced
morphism $f^*\text{Ker}(a_V)\to\text{Ker}(a_U)$ is an
isomorphism:
$$
\xymatrix@C=.8cm @R=.7cm{
f^*\text{Ker}(a_V) \ar[r] \ar[d] & f^*F_V
\ar[r]^-{f^*a_V} \ar[d]^-{\sim} & f^*G_V \ar[d]^-{\sim}
\\
\text{Ker}(a_U) \ar[r] & F_U \ar[r]^-{a_U} & G_U.}
$$
The proof for cokernels and extensions (using 5-lemma)
is similar. One can also mimic the proof in (\cite{Ols3},
3.8, 3.9) to prove a similar statement for analytic stacks,
in the more general setting where the coefficient ring is
a \textit{flat sheaf}. In this paper, we will only need
the case of a constant coefficient ring.

Then by (\ref{L-analytic-constr} iii), it suffices to show
that lcc $R$-modules form a Serre subcategory. This follows
from (\cite{SGA4}, IX, 2.1).
\end{anitem}

\subsection{Derived categories}\label{analytic-derived}

\begin{anitem}\label{LO2-analytic}
Again let $\fk X=\mathcal X^{\text{an}}.$ We follow
\cite{LO2} and define the \textit{derived category
$D_c(\fk X_{\emph{lis-an}},\Lambda)$
of constructible $\Lambda$-adic sheaves} (by abuse of
language, as usual) as follows.

A complex of projective systems $M$ in the ordinary derived
category $\s D(\fk X_{\text{lis-an}}^{\bb N},\Lambda_{\bullet})$
of the simplicial topos $\fk X_{\text{lis-an}}^{\bb N}$ ringed
by $\Lambda_{\bullet}=(\Lambda_n)_n,$ is called a
\textit{$\lambda$-complex} if for every $i$ and $n,$ the
sheaf $\s H^i(M_n)$ is constructible and the cohomology
system $\s H^i(M)$ is AR-adic. A \textit{$\lambda$-module}
is a $\lambda$-complex concentrated in degree 0, i.e.
the $\s H^i$'s are AR-null for $i\ne0.$ Then we define
$D_c(\fk X_{\text{lis-an}},\Lambda)$ to be the quotient
of the full subcategory $\s D_c(\fk X_{\text{lis-an}}^{\bb
N},\Lambda_{\bullet})$ of $\lambda$-complexes by the full
subcategory of AR-null complexes
(i.e. those with AR-null cohomology systems).

This quotient inherits a standard $t$-structure, and we define
the category $\Lambda\text{-Sh}_c(\fk X)$ of
\textit{constructible $\Lambda$-adic sheaves} on
$\fk X_{\text{lis-an}}$ to be its core, namely the
quotient of the category of $\lambda$-modules by
the thick full subcategory of AR-null systems. By
(\cite{SGA5}, p.234), this is equivalent to the category
of adic systems, i.e.\ those projective systems
$F=(F_n)_n,$ such that for each $n,\ F_n$ is a
constructible $\Lambda_n$-module on $\fk
X_{\text{lis-an}},$ and the induced morphism
$F_n\otimes_{\Lambda_n}\Lambda_{n-1}\to F_{n-1}$ is an
isomorphism.

Passing to localizations and 2-colimits, one can
also define the categories $D_c(\fk X_{\text{lis-an}},
E_{\lambda})$ and $D_c(\fk X_{\text{lis-an}},\overline{\bb
Q}_{\ell}),$ as well as their cores with respect to
the standard $t$-structures: the
categories of constructible $E_{\lambda}$- and $\overline{\bb
Q}_{\ell}$-sheaves on $\fk X_{\text{lis-an}}.$
\end{anitem}

\begin{anitem}\label{cart-simplicial}
Let $X_{\bullet}$ be a strictly simplicial analytic
space. Then we have the subcategory $\text{Ab}
_{\text{cart}}(X_{\bullet})$ of \textit{Cartesian
abelian sheaves on $X_{\bullet}$} in $\text{Ab}(X_{\bullet}),$
as defined in (\cite{LMB}, 12.4.2). It can be proved
in the same way as in (\ref{constr=serre}) that this
is a Serre subcategory, which enables us to
define the triangulated subcategory $\s D_{\text{cart}}
(X_{\bullet},\bb Z)$ of the ordinary derived category $\s
D(X_{\bullet},\bb Z),$ consisting of complexes with
Cartesian cohomology sheaves.

If $\fk X$ is an analytic stack and $X_{\bullet}\to\fk
X$ is a strictly simplicial hypercover of $\fk X$ by
analytic spaces, one can also consider the localized
topos $\fk X_{\text{lis-an}}|_{X_{\bullet}}$ (cf.\
(\cite{SGA4}, IV, 5)) and have the notion of \textit
{Cartesian sheaves} on it. Then $\text{Ab}_{\text
{cart}}(\fk X|_{X_{\bullet}})\subset\text{Ab}(\fk
X|_{X_{\bullet}})$ is a Serre subcategory, and
we may define the triangulated subcategory $\s
D_{\text{cart}}(\fk X|_{X_{\bullet}},\bb Z)\subset
\s D(\fk X|_{X_{\bullet}},\bb Z).$ These constructions
apply as well to a general coefficient ring $R$ in place
of $\bb Z.$

Now let $X_{\bullet}$ be a strictly simplicial $\bb
C$-algebraic space and $R$ be a noetherian ring. A sheaf $F$ of sets (resp.\
$R$-modules) on $X_{\bullet}(\bb C)$ is said to be
\textit{constructible} if it is Cartesian and
all components $F_n$ on $X_n(\bb C)$ are constructible.
Constructible $R$-modules on $X_{\bullet}(\bb C)$
form a Serre subcategory, and one can define the triangulated
subcategory $\s D_c(X_{\bullet}(\bb C),R)$ consisting
of complex with constructible cohomology sheaves.
When $X_{\bullet}\to\cal X$ is a strictly simplicial
hypercover of a complex algebraic stack $\cal X$ and
$\fk X=\mathcal X^{\text{an}},$ we also have $\s
D_c(\fk X|_{X_{\bullet}(\bb C)},R).$
Following (\cite{LO2}, 10.0.6), we define the adic derived
category $D_c(X_{\bullet}(\bb C),\Lambda)$ as follows.

A sheaf $F\in\text{Mod}(X_{\bullet}(\bb C)^{\bb N},
\Lambda_{\bullet})$ is \textit{AR-adic} if it is Cartesian
(i.e.\ each $F_n\in\text{Mod}(X_{\bullet}(\bb C),\Lambda_n)$
is Cartesian), and $F|_{X_i^{\bb N}}$ is AR-adic
for every $i.$ A complex $C\in\s D(X_{\bullet}(\bb C)
^{\bb N},\Lambda_{\bullet})$ is a \textit
{$\lambda$-complex} (resp.\ an \textit{AR-null complex})
if the cohomology sheaf $\s H^i(C)$ is AR-adic and
$\s H^i(C_m)|_{X_n(\bb C)}\in\text{Mod}(X_n(\bb C),\Lambda_m)$ is constructible,
for every $i,m,n$ (resp.\ $C|_{X_n(\bb C)}$ is
AR-null, for every $n$). Finally we define
$D_c(X_{\bullet}(\bb C),\Lambda)$ to be the quotient
of the full subcategory $\s D_c(X_{\bullet}(\bb C)
^{\bb N},\Lambda_{\bullet})\subset\s
D(X_{\bullet}(\bb C)^{\bb N},\Lambda_{\bullet})$
consisting of all $\lambda$-complexes by the full
subcategory of AR-null complexes.
\end{anitem}

\begin{anitem}\label{complex-derived}
Let $R$ be a noetherian ring as before. Then we may also define
a topological constructible derived category as follows.
Let $\s D(\fk X_{\text{lis-an}},R)$ be the ordinary
derived category of sheaves of $R$-modules on $\fk X_{\text{lis-an}}$.
Then (\ref{constr=serre}) allows
us to define its triangulated subcategories
$$
\s D_c(\fk X_{\text{lis-an}},R)\subset\s D_{\text
{cart}}(\fk X_{\text{lis-an}},R),
$$
consisting of complexes with constructible
cohomology sheaves and Cartesian cohomology sheaves,
respectively. The cores of the standard $t$-structures
on them are $\text{Mod}_c(\fk X,R)$ and $\text{Mod}
_{\text{cart}}(\fk X,R)$ respectively. The main
examples we have in mind of the ring $R$ for this topological
setting are $\Lambda,\bb Q,\bb C$ and $\overline{\bb Q}_{\ell}.$

In particular, when $R=\Lambda,$ to emphasize the difference
from the category $\Lambda\text{-Sh}_c(\fk X)$ of
$\Lambda$-adic sheaves on $\fk X,$ we will often denote
by $\fk{Mod}_c(\Lambda_{\fk X})$ the category of
constructible $\Lambda_{\fk X}$-modules. In (\ref{agree}),
we will show that the two categories
$D_c(\fk X_{\text{lis-an}},\Lambda)$ and $\s
D_c(\fk X_{\text{lis-an}},\Lambda)$ are equivalent.

For simplicity, we will drop ``lis-an" in
$\s D_c(\fk X_{\text{lis-an}},R),$ if there is no
confusion. Also we will drop ``lis-\'et" in $D_c(\mathcal
X_{\text{lis-\'et}},R).$
\end{anitem}

\subsection{Comparison between the derived categories of
lisse-\'etale and lisse-analytic topoi}\label{compar}

Given an algebraic stack $\mathcal X/\bb C,$ let
$\fk X=\mathcal X^{\text{an}},$ and let $P:X\to
\mathcal X$ be a presentation, with analytification
$P^{\text{an}}:X(\bb C)\to\fk X.$ Let
$\epsilon:X_{\bullet}\to\mathcal X$ be the associated
strictly simplicial smooth hypercover, and let
$\epsilon^{\text{an}}:X_{\bullet}(\bb C)\to\fk
X$ be the analytification. They induce morphisms of
topoi, denoted by the same symbol. Consider the following
morphisms of topoi:
$$
\xymatrix@C=.8cm @R=.8cm{
\fk X_{\text{lis-an}} & \fk X_{\text{lis-an}}
|_{X_{\bullet}(\bb C)} \ar[r]^-{\delta^{\text
{an}}_{\bullet}} \ar[l]_-{\gamma^{\text{an}}} &
X_{\bullet}(\bb C) \ar[d]^-{\xi_{\bullet}}
\ar@/^1.7pc/[ll]_-{\epsilon^{\text{an}}} \\
\mathcal X_{\text{lis-\'et}} & \mathcal
X_{\text{lis-\'et}}|_{X_{\bullet}}
\ar[r]^-{\delta_{\bullet}} \ar[l]_-{\gamma} &
X_{\bullet,\text{\'et}} \ar@/^1.7pc/[ll]_-{\epsilon}.}
$$
We will show that
$R\epsilon_*\circ R\xi_{\bullet,*}\circ\epsilon^{\text
{an},*}$ gives an equivalence between $D_c(\fk
X,\Lambda)$ and $D_c(\mathcal X,\Lambda),$ and that it is
compatible with pushforwards. It is proved in
(\cite{LO1}, 2.2.6) that, $(\epsilon^*,R\epsilon_*)$
induce an equivalence between the triangulated categories
$\s D_c(\mathcal X,\Lambda_n)$ and $\s D_c(X_{\bullet},
\Lambda_n).$ As in (\cite{LO2}, 10.0.8), this gives an
equivalence between $D_c(\mathcal X,\Lambda)$ and
$D_c(X_{\bullet,\text{\'et}},\Lambda).$
We mimic the proof there and prove the analytic
analogue.

\begin{subproposition}\label{P-compar}
\emph{(i)} Let $R$ be a noetherian commutative ring with identity.
Then the pairs of functors $(\epsilon^{\emph{an},*},
R\epsilon_*^{\emph{an}}),(\delta^{\emph{an},*}_{\bullet},
R\delta^{\emph{an}}_{\bullet,*})$ and $(\gamma^{\emph{an},*},
R\gamma^{\emph{an}}_*)$ induce equivalences of triangulated
categories
$$
\xymatrix@C=.5cm @R=.5cm{
& \s D_c(\fk X|_{X_{\bullet}(\bb C)},R)
\ar@{<->}[rd] \ar@{<->}[ld] & \\
\s D_c(\fk X,R) \ar@{<->}[rr] && \s D_c(X_{\bullet}(\bb C),R)}
$$
which is commutative, as well as an equivalence
$$
\xymatrix@C=1cm{
D_c(\fk X,\Lambda) \ar@{<->}[r] &
D_c(X_{\bullet}(\bb C),\Lambda).}
$$

\emph{(ii)} Let $X$ be a $\bb C$-algebraic space, and let
$\xi=\xi_X:X(\bb C)\to X_{\emph{\'et}}$ be the natural
morphism of topoi. Then $R\xi_*$ is defined on the
unbounded constructible derived category, and the functors
$(\xi^*,R\xi_*)$ induce an equivalence between
$D_c(X(\bb C),\Lambda)$ and $D_c(X,\Lambda).$

\emph{(iii)} Let $f:X\to Y$ be a morphism of $\bb C$-algebraic
spaces, and let $\xi_X,\xi_Y$ be as in (ii). Then for
every $F\in D_c^+(X,\Lambda),$ the natural morphism
$$
\xi_Y^*f_*F\to f^{\emph{an}}_*(\xi_X^*F)
$$
is an isomorphism. Recall (\ref{nota-derived}) that $f_*$
and $f_*^{\emph{an}}$ here are derived functors $Rf_*$ and
$Rf_*^{\emph{an}},$ respectively.
\end{subproposition}

\begin{proof}
(i) Firstly, note that the restriction
functor $\delta^{\text{an}}_{\bullet,*}:
\text{Ab}(\fk X_{\text{lis-an}}|_{X_{\bullet}(\bb
C)})\to\text{Ab}(X_{\bullet}(\bb C))$ is exact so that
$R\delta^{\text{an}}_{\bullet,*}=\delta^{\text{an}}_{\bullet,*},$
since the topologies are the same. We have $\epsilon^{
\text{an},*}\simeq\delta^{\text{an}}_{\bullet,*}\circ
\gamma^{\text{an},*},$ since they are all restrictions.
Therefore, it suffices to prove that $(\delta^{\text{an},*}_{\bullet},
R\delta^{\text{an}}_{\bullet,*})$ and $(\gamma^{\text{an},*},
R\gamma^{\text{an}}_*)$ induce equivalences of triangulated
categories.

For an abelian sheaf $F$ on $X_{\bullet}(\bb C),$ given by
$F_n\in\text{Ab}(X_n(\bb C))$ for each $n$ and the
transition map $\theta_a:a^*F_n\to F_m$ for each morphism
$a:m\to n$ in $\Delta^{+,\text{op}}$ (i.e.\ for each order-preserving injection $a:\{0,\cdots,n\}\to\{0,\cdots,m\}$), the sheaf
$\delta^{\text{an},*}_{\bullet}F$ assigns to the object
$$
\xymatrix@C=.7cm @R=.5cm{
U \ar[r]^-u \ar[rd] & X_n(\bb C) \ar[d] \\
& \fk X}
$$
the sheaf $u^*F_n$ on $U_{\text{an}},$ and to each
morphism
$$
\xymatrix@C=.7cm @R=.5cm{
U' \ar[r]^-{u'} \ar[d]_-{\varphi} & X_m(\bb C) \ar[d]^-a \\
U \ar[r]^-u \ar[rd] & X_n(\bb C) \ar[d] \\
& \fk X}
$$
the transition map $\theta_{a,\varphi}:$
$$
\xymatrix@C=.7cm @R=.5cm{
\varphi^*u^*F_n \ar[r]^-{\theta_{a,\varphi}}
\ar[d]_-{\sim} & u'^*F_m \\
u'^*a^*F_n \ar[ru]_-{u'^*\theta_a} &}.
$$
It is then clear that
$(\delta_{\bullet}^{\text{an},*},\delta_{\bullet,*}
^{\text{an}})$ induce equivalences of categories
$$
\text{Ab}_{\text{cart}}(\fk X|_{X_{\bullet}(\bb C)})
\longleftrightarrow\text{Ab}_{\text{cart}}(X_{\bullet}(\bb
C)),\qquad\text{Ab}_c(\fk X|_{X_{\bullet}(\bb C)})
\longleftrightarrow\text{Ab}_c(X_{\bullet}(\bb C)).
$$
For $F\in\s D_{\text{cart}}(X_{\bullet}(\bb C),\bb Z)$ and
$G\in\s D_{\text{cart}}(\fk X|_{X_{\bullet}(\bb C)},
\bb Z),$ we see that the adjunction and coadjunction
morphisms
$$
F\to\delta^{\text{an}}_{\bullet,*}\delta^{\text{an},*}
_{\bullet}F,\qquad
\delta^{\text{an},*}_{\bullet}\delta^{\text{an}}_{\bullet,*}G\to G
$$
are isomorphisms by applying $\s H^i.$ Hence
$(\delta_{\bullet}^{\text{an},*},\delta_{\bullet,*}
^{\text{an}})$ induce equivalences
$$
\s D_{\text{cart}}(\fk X|_{X_{\bullet}(\bb C)},\bb Z)
\longleftrightarrow\s D_{\text{cart}}(X_{\bullet}(\bb
C),\bb Z),\quad\s D_c(\fk X|_{X_{\bullet}(\bb C)},\bb Z)
\longleftrightarrow\s D_c(X_{\bullet}(\bb C),\bb Z),
$$
and also with $\bb Z$ replaced by any noetherian ring $R.$

To show that $\gamma^{\text{an},*}$ induces an equivalence
with coefficient $R,$ we will apply (\cite{LO1},
2.2.3). All the transition morphisms of topoi in the strictly
simplicial ringed topos $(\fk X_{\text{lis-an}}|_{X_{\bullet}
(\bb C)},R)$ as well as $\gamma^{\text{an}}:(\fk
X_{\text{lis-an}}|_{X_{\bullet}(\bb C)},R)\to(\fk
X_{\text{lis-an}},R)$ are flat. Let $\s C=\text{Mod}
_c(\fk X,R),$ which is a Serre subcategory of
$\text{Mod}(\fk X,R)$ by (\ref{constr=serre}), and let
$\s C_{\bullet}$ be the essential image of $\s C$ under
$\gamma^{\text{an},*}:\text{Mod}(\fk X,R)\to\text{Mod}(\fk
X|_{X_{\bullet}(\bb C)},R).$ We will see shortly that $\s
C_{\bullet}=\text{Mod}_c(\fk X|_{X_{\bullet}(\bb C)},R).$
To apply (\cite{LO1}, 2.2.3) we need to verify the
assumption (\cite{LO1}, 2.2.1), which has two parts:

$\bullet$ (\cite{LO1}, 2.1.7) for the ringed sites
$(\text{Lis-an}(\fk X)|_{X_i(\bb C)},R)$ with $\s C_i$
the essential image of $\s C$ under the restriction
$\text{Mod}(\fk X,R)\to\text{Mod}(\fk X|_{X_i(\bb C)},R).$
This means that, for every object $U$ in this site,
there exists an analytic open covering $U=\cup
U_{\alpha}$ and an integer $n_0,$ such that for every
$F\in\s C_i$ and $n\ge n_0,$ we have $H^n(U_{\alpha},F)=0.$
This follows from (\cite{Dim}, 3.1.7, 3.4.1).

$\bullet\ \gamma^{\text{an},*}:\s C\to\s C_{\bullet}$
is an equivalence with quasi-inverse $R\gamma^{\text{an}}_*.$
For $F\in\text{Mod}(\fk X,R),$ its image
$\gamma^{\text{an},*}F$ is the sheaf that assigns to the object
$$
\xymatrix@C=.7cm @R=.5cm{
U \ar[r]^-u \ar[rd] & X_n(\bb C) \ar[d] \\
& \fk X}
$$
the sheaf $u^*F_{X_n(\bb C)},$ and to each morphism
$$
\xymatrix@C=.7cm @R=.5cm{
U' \ar[r]^-{u'} \ar[d]_-{\varphi} & X_m(\bb C) \ar[d]^-a \\
U \ar[r]^-u \ar[rd] & X_n(\bb C) \ar[d] \\
& \fk X}
$$
the transition map $\theta_{a,\varphi}:$
$$
\xymatrix@C=.7cm @R=.5cm{
\varphi^*u^*F_{X_n(\bb C)} \ar[r]^-{\theta_{a,\varphi}}
\ar[d]_-{\sim} & u'^*F_{X_m(\bb C)} \\
u'^*a^*F_{X_n(\bb C)} \ar[ru]_-{u'^*\theta_a} &}.
$$
So it is clear that $\gamma^{\text{an},*}$ sends
Cartesian (resp.\ constructible) sheaves to Cartesian
(resp.\ constructible) sheaves. To verify this assumption,
we need the analytic version of (\cite{Ols3}, 4.4, 4.5),
which we state in the following for the reader's
convenience.

Let $\text{Des}(X(\bb C)/\fk X,R)$ be
the category of pairs $(F,\alpha),$ where
$F\in\text{Mod}(X(\bb C),R),$ and
$\alpha:p_1^*F\to p_2^*F$ is an isomorphism on
$X_1(\bb C)$ (where $p_1$
and $p_2$ are the natural projections $X_1(\bb C)
\rightrightarrows X_0(\bb C)=X(\bb C)$), such
that $p_{13}^*(\alpha)=p_{23}^*(\alpha)\circ
p_{12}^*(\alpha):\bar{p}_1^*F\to\bar{p}_3^*F$ on
$X_2(\bb C).$ Here $\bar{p}_i:X_2\to X_0$ are the
natural projections. There is a natural functor $A:\text
{Mod}_{\text{cart}}(\fk X,R)\to\text{Des}
(X(\bb C)/\fk X,R),$ sending $M$ to
$(F,\alpha),$ where $F=M_{X(\bb C)}$ and $\alpha$
is the composite
$$
\xymatrix@C=1cm{
p_1^*F \ar[r]^-{p_1^*} & M_{X_1(\bb C)} \ar[r]^-
{(p_2^*)^{-1}} & p_2^*F.}
$$
There is also a natural functor
$B:\text{Mod}_{\text{cart}}(X_{\bullet}(\bb C),
R)\to\text{Des}(X(\bb C)/\fk X,R)$ sending $F=(F_i)_i$ to
$(F_0,\alpha),$ where $\alpha$ is the composite
$$
\xymatrix@C=1cm{
p_1^*F_0 \ar[r]^-{\text{can}} & F_1
\ar[r]^-{\text{can}^{-1}} & p_2^*F_0,}
$$
and the cocycle condition is verified as in (\cite{Ols3},
4.5.4).

\begin{sublemma}\label{Ols3-4.4}
The natural functors in the diagram
$$
\xymatrix@C=.7cm @R=.8cm{
& \emph{Mod}_{\emph{cart}}(X_{\bullet}(\bb C),
R) \ar[dr]^-B & \\
\emph{Mod}_{\emph{cart}}(\fk X,R) \ar[rr]^-A
\ar[ur]^-{\emph{res}} && \emph{Des}(X(\bb
C)/\fk X,R)}
$$
are all equivalences, and the diagram is commutative up
to natural isomorphism.
\end{sublemma}

The proof in (\cite{Ols3}, 4.4, 4.5) carries verbatim
to analytic stacks. In particular, by (\ref{L-analytic-constr})
the restriction
$$
\text{Mod}_c(\fk X,R)\to\text{Mod}_c(X_{\bullet}(\bb C),R)
$$
is an equivalence.

Note that $\s
C_{\bullet}=\text{Mod}_c(\fk X|_{X_{\bullet}(\bb C)},R).$
Clearly every object in $\s C_{\bullet}$ is constructible.
Conversely, for any constructible $R$-module $G_{\bullet}$ on
$\fk X|_{X_{\bullet}(\bb C)},$ we have $G_{\bullet}\cong\delta_{\bullet}^{\text{an},*}F_{\bullet}$
where $F_{\bullet}$ is the restriction
$\delta_{\bullet,*}^{\text{an}}G_{\bullet}\in
\text{Mod}_c(X_{\bullet}(\bb C),R)$ of $G_{\bullet}.$ By
(\ref{Ols3-4.4}) we see that $F_{\bullet}$ is the
restriction of $F\in\text{Mod}_c(\fk X,R)$ for a unique
(up to isomorphism) constructible $R_{\fk X}$-module $F,$
therefore $G_{\bullet}\cong\gamma^{\text{an},*}
F$ from that $\delta^{\text{an}}_{\bullet,*}G_{\bullet}
\cong F_{\bullet}\cong\text{res}\ F=\delta^{
\text{an}}_{\bullet,*}(\gamma^{\text{an},*}F),$ and
hence $G_{\bullet}\in\s C_{\bullet}.$

From the 2-commutative diagram
$$
\xymatrix@C=.7cm @R=.8cm{
& \text{Mod}_c(\fk X|_{X_{\bullet}(\bb C)},
R) \ar[dr]^-{\delta^{\text{an}}_{\bullet,*}}_-{\sim} & \\
\text{Mod}_c(\fk X,R) \ar[ur]^-{\gamma^{\text{an},*}}
\ar[rr]^-{\text{res}}_-{\sim} && \text{Mod}_c(
X_{\bullet}(\bb C),R)}
$$
we see that $\gamma^{\text{an},*}:\s C\to\s C_{\bullet}$
is an equivalence.

By (\cite{LO1}, 2.2.3), the functors $(\gamma^{\text{an},
*},R\gamma^{\text{an}}_*)$ induce an equivalence
$$
\s D_c(\fk X,R)\longleftrightarrow\s D_c(\fk
X|_{X_{\bullet}(\bb C)},R).
$$

Note that for $M\in\s D_c(\fk X^{\bb
N},\Lambda_{\bullet})$ (resp.\ $\s D_c(X_{\bullet}(\bb
C)^{\bb N},\Lambda_{\bullet})$), each level $M_n$ is in
$\s D_c(\fk X,\Lambda_n)$ (resp.\ $\s D_c(X_{\bullet}(\bb C),
\Lambda_n)$), and the property of $M$ being
AR-adic (resp.\ AR-null) is intrinsic (\cite{SGA5}, V, 3.2.3).
So AR-adic (resp.\ AR-null) complexes on the two
sides correspond under this equivalence, and
we have equivalences
$$
\s D_c(\fk X^{\bb N},\Lambda_{\bullet})
\longleftrightarrow\s D_c(X_{\bullet}(\bb C)^{\bb
N},\Lambda_{\bullet}),\qquad
D_c(\fk X,\Lambda)\longleftrightarrow D_c(X_{\bullet}(\bb
C),\Lambda).
$$

(ii) We prove it for torsion coefficients first, and
then pass to adic coefficients. For torsion coefficients,
we prove it for schemes first, and then apply descent
(\cite{LO1}, 2.2.3) to deduce it for algebraic spaces.

Let $X/\bb C$ be a scheme and let $G$ be
a sheaf of abelian groups on $X(\bb C).$ Then the
sheaf $R^i\xi_*G$ on $X_{\text{\'et}}$ is the
sheafification of the presheaf
$$
(U\to X)\mapsto H^i(U(\bb C),G).
$$
By (\cite{Dim}, 3.1.7, 3.4.1), $R^i\xi_*G=0$ for all
sheaves $G$ and all $i>1+2\dim X,$ so $R\xi_*$
has finite cohomological dimension, and the functor
$$
R\xi_*:\s D(X(\bb C),\Lambda_n)\to\s D(X,\Lambda_n)
$$
takes the full subcategory $\s D_c(X(\bb C),
\Lambda_n)$ into $\s D_c(X,\Lambda_n).$

Given $F\in\s D_c(X,\Lambda_n)$ and $G\in\s D_c(X(\bb
C),\Lambda_n),$ we want to show that the adjunction
and coadjunction morphisms
$$
F\to R\xi_*\xi^*F,\qquad \xi^*R\xi_*G\to G
$$
are isomorphisms. The analytification functor $\xi^*$ is
exact and $\text{cd}(R\xi_*)<\infty,$ so by
applying $\s H^i$ on both sides,
we may assume that $F$ and $G$ are bounded, or even
constructible sheaves. By the comparison (\cite{BBD},
6.1.2 $(A')$), $G$ is algebraic (i.e. $G=\xi^*\widetilde{G}$
for some $\Lambda_n$-sheaf $\widetilde{G}$ on
$X_{\text{\'et}}$).

The sheaves $R^i\xi_*\xi^*F$ and $R^i\xi_*G$ are
sheafifications of the functors on $\text{\'Et}(X)$
$$
(U\to X)\mapsto H^i(U(\bb C),F^{\text{an}}),\qquad
(U\to X)\mapsto H^i(U(\bb C),G)
$$
respectively. By the comparison theorem of Artin
(\cite{SGA4}, XVI, 4.1), we have
$$
H^i(U(\bb C),F^{\text{an}})=H^i(U,F),\qquad
H^i(U(\bb C),G)=H^i(U,\widetilde{G}),
$$
so they both sheafify to zero if $i>0$ (\cite{Mil1}, 10.4),
and to $F$ and $\widetilde{G}$ respectively if
$i=0.$ It follows that the adjunction and coadjunction
morphisms are both isomorphisms, and we have an equivalence
$$
(\xi^*,R\xi_*):\s D_c(X(\bb C),\Lambda_n)\longleftrightarrow
\s D_c(X,\Lambda_n)
$$
for each $n.$

Now let $X$ be a $\bb C$-algebraic space, and take a simplicial
\'etale hypercover $\epsilon:X_{\bullet}\to X$ of $X$ by
schemes. As in (i), we can apply (\cite{LO1}, 2.2.3) to
show that the morphisms of topoi
$$
\xymatrix@C=.5cm{
X_{\bullet,\text{\'et}} & X_{\text{\'et}}|_{X_{\bullet}}
\ar[l] \ar[r] & X_{\text{\'et}}}
$$
induce equivalences
$$
\s D_c(X_{\bullet,\text{\'et}},\Lambda_n)\longleftrightarrow
\s D_c(X_{\text{\'et}}|_{X_{\bullet}},\Lambda_n)\longleftrightarrow
\s D_c(X_{\text{\'et}},\Lambda_n)
$$
for each $n.$ Similarly, $\epsilon^{\text{an}}:X_{\bullet}
(\bb C)\to X(\bb C)$ induces an equivalence
$$
\s D_c(X_{\bullet}(\bb C),\Lambda_n)\to\s D_c(X(\bb C),\Lambda_n).
$$
Therefore, the commutative diagram of topoi
$$
\xymatrix@C=.8cm @R=.6cm{
X_{\bullet}(\bb C) \ar[r]^-{\epsilon^{\text{an}}}
\ar[d]_-{\xi_{X_{\bullet}}} & X(\bb C) \ar[d]^-{\xi_X} \\
X_{\bullet,\text{\'et}} \ar[r]_-{\epsilon} &
X_{\text{\'et}}}
$$
leads to a commutative diagram
$$
\xymatrix@C=.8cm @R=.6cm{
\s D_c(X_{\bullet}(\bb C),\Lambda_n) \ar[r]^-{\sim} \ar[d]^-{\sim}
& \s D_c(X(\bb C),\Lambda_n) \ar[d]^-{(1)} \\
\s D_c(X_{\bullet,\text{\'et}},\Lambda_n) \ar[r]^-{\sim} &
\s D_c(X_{\text{\'et}},\Lambda_n),}
$$
and (1) is an equivalence.

Since AR-adic (resp.\ AR-null) complexes in $\s D(X(\bb C)^{\bb
N},\Lambda_{\bullet})$ and in $\s D(X^{\bb N},\Lambda_{\bullet})$
correspond, $(\xi^*,R\xi_*)$ induce equivalences
$$
\s D_c(X(\bb C)^{\bb N},\Lambda_{\bullet})
\longleftrightarrow\s D_c(X^{\bb N},\Lambda_{\bullet}),
\qquad D_c(X(\bb C),\Lambda)\longleftrightarrow
D_c(X,\Lambda).
$$

(iii) Applying $\s H^i$ on both sides, we need to show that
$$
\xi_Y^*R^if_*F\to R^if^{\text{an}}_*(\xi_X^*F)
$$
is an isomorphism. The normalization functor $F\mapsto\widehat{F}$
has finite cohomological dimension, so $\widehat{F}$
is essentially bounded below. Replacing $F$ by various levels
$\widehat{F}_n$ of its normalization, we reduce to the
case where $F\in\s D_c^+(X,\Lambda_n).$ Then one
can replace $F$ by $\tau_{\le i}F$ and reduce to the case
where $F$ is bounded, or even a constructible
$\Lambda_n$-sheaf. For schemes, this follows from
Artin's comparison theorem (\cite{SGA4}, XVI, 4.1).
For algebraic spaces, this follows from
the case of schemes by descent.

Explicitly, to prove that the base change morphism is an
isomorphism, we may pass to an \'etale presentation of
$Y$ and apply smooth base change theorem, hence reduce
to the case where $Y$ is a scheme. Then let
$X_{\bullet}\to X$ be a simplicial \'etale cover by
schemes, and let $f_p:X_p\to Y\ (p\ge1)$ be the composition
$X_p\to X\overset{f}{\to}Y.$ The commutative diagram
of topoi
$$
\xymatrix@C=.8cm @R=.6cm{
X_{\bullet}(\bb C) \ar[r] \ar[d]_-{\xi_{X_{\bullet}}}
& X(\bb C) \ar[r]^-{f^{\text{an}}} \ar[d]_-{\xi_X} &
Y(\bb C) \ar[d]^-{\xi_Y} \\
X_{\bullet,\text{\'et}} \ar[r] & X_{\text{\'et}}
\ar[r]^-f & Y_{\text{\'et}}}
$$
leads to a morphism of spectral sequences
$$
\xymatrix@C=1cm @R=.7cm{
& \xi_Y^*R^qf_{p*}(F|_{X_p}) \ar@{=>}[r] \ar[ld]_-{(1)}
& \xi_Y^*R^{p+q}f_*F \ar[d]^-{(2)} \\
R^qf^{\text{an}}_{p*}(\xi_{X_p}^*(F|_{X_p}))
\ar[r]^-{\sim} & R^qf^{\text{an}}_{p*}((\xi_X^*F)|_{X_p(\bb
C)}) \ar@{=>}[r] & R^{p+q}f^{\text{an}}_*(\xi_X^*F),}
$$
where (1) is an isomorphism, therefore (2) is an
isomorphism.
\end{proof}

\begin{subcorollary}\label{compar-cohom}
There is a natural equivalence between the triangulated
categories $D_c(\fk X,\Lambda)$ and $D_c(\mathcal
X,\Lambda),$ compatible with pushforwards by morphisms
of complex algebraic stacks.
\end{subcorollary}

\begin{proof}
As mentioned before, combining (\cite{LO1}, 2.2.6) and
(\ref{P-compar}) we see that $R\epsilon_*\circ
R\xi_{\bullet,*}\circ\epsilon^{\text{an},*}$ gives an
equivalence between $D_c(\fk X,\Lambda)$ and
$D_c(\mathcal X,\Lambda).$ Let $f:\mathcal X\to\mathcal Y$
be a morphism of $\bb C$-algebraic stacks, and we
choose a commutative diagram
$$
\xymatrix@C=.8cm @R=.6cm{
X_{\bullet} \ar[r]^-{\widetilde{f}} \ar[d]_-{\epsilon_{
\mathcal X}} & Y_{\bullet} \ar[d]^-{\epsilon_{\mathcal
Y}} \\
\mathcal X \ar[r]^-f & \mathcal Y.}
$$
Then we have the following diagram
$$
\xymatrix@C=.8cm @R=.7cm{
D_c^+(\fk X,\Lambda) \ar[r]^-{\epsilon_{\mathcal
X}^{\text{an},*}} \ar[d]^-{f^{\text{an}}_*} &
D_c^+(X_{\bullet}(\bb C),\Lambda)
\ar[d]^-{\widetilde{f}^{\text{an}}_*} &
D_c^+(X_{\bullet},\Lambda) \ar[l]_-{\xi_{X_{\bullet}}^*}
\ar[d]^-{\widetilde{f}_*} & D_c^+(\mathcal X,\Lambda)
\ar[l]_-{\epsilon_{\mathcal X}^*} \ar[d]^-{f_*} \\
D_c^+(\fk Y,\Lambda) \ar[r]_-{\epsilon_{\mathcal
Y}^{\text{an},*}} & D_c^+(Y_{\bullet}(\bb C),
\Lambda) & D_c^+(Y_{\bullet},\Lambda)
\ar[l]^-{\xi_{Y_{\bullet}}^*} & D_c^+(\mathcal
Y,\Lambda), \ar[l]^-{\epsilon_{\mathcal Y}^*}}
$$
where the horizontal arrows are all equivalences of
triangulated categories. The square on the left commutes
by construction, and the commutativity of the squares in
the middle and on the right follows from (\ref{P-compar} iii)
and (\cite{LO2}, p.202) respectively.

It remains to show that the equivalence is ``natural" in the
sense that, if $P':X'\to\mathcal X$ is another presentation,
the induced equivalence is naturally isomorphic to the one
induced by $X.$ The usual argument of taking 2-fiber product
reduces us to assume that one presentation dominates the
other, and the claim is clear in this case.
\end{proof}

\subsection{Comparison between the two derived
categories on the lisse-analytic
topos}\label{conts-derived}

In (\ref{LO2-analytic}) and (\ref{complex-derived}), we
defined two derived categories, denoted by $D_c(\fk
X,\Lambda)$ and $\s D_c(\fk X,\Lambda)$
respectively. Before proving that they are equivalent, we give
some preparation on the analytic analogues of some
concepts and results in \cite{LO2}.

\begin{anitem}\label{analytic-normalization}
For now let $\fk X$ be any complex analytic stack,
not necessarily algebraic. As in \cite{Eke}, let $\pi:(\fk
X^{\bb N}_{\text{lis-an}},\Lambda_{\bullet})\to(\fk
X_{\text{lis-an}},\Lambda)$ be the morphism of ringed topoi, with
$\pi_*=\varprojlim$ and $\pi^*=(-\otimes_{\Lambda}\Lambda_n)_n.$
We then have derived functors $R\pi_*$ and $L\pi^*$ between $\s
D(\fk X^{\bb N},\Lambda_{\bullet})$ and $\s D(\fk X,\Lambda),$
and $L\pi^*$ has (co)homological dimension 1.
Denote $\text{Mod}(\fk X^{\bb N},\Lambda_{\bullet})$ by $\s
A(\fk X)$ or just $\s A.$

Variant: more generally, one can consider other coefficients
for the topoi $\fk X^{\bb N}_{\text{lis-an}}$ and $\fk X_{\text{lis-an}},$ for instance the
constant sheaf $\bb Z,$ namely one is considering all sheaves
of abelian groups. In this case we have derived functors
$R\pi_*$ and $L\pi^*=L\pi^{-1}=\pi^{-1}$ between $\s
D(\fk X^{\bb N},\bb Z)$ and $\s D(\fk X,\bb Z).$

When $\fk X=X$ is an analytic space (always assumed
to be finite dimensional), we often consider
the morphism $\pi_X:X_{\text{an}}^{\bb N}\to
X_{\text{an}}$ between analytic topoi. The derived
functor $R\pi_{X*}:\s D(X^{\bb N},\bb Z)\to\s
D(X,\bb Z)$ has finite cohomological dimension: this is a consequence
of (\cite{LO2}, 2.1.i) and (\cite{Dim}, 3.1.7, 3.4.1). In this
case, by $\pi_X$ we always mean this morphism between analytic
topoi, rather than the lisse-analytic topoi of $X,$ unless
otherwise stated.

It follows from (\ref{LO2-2.2.1} i) below that, for
an analytic stack $\fk X$ with a presentation $X,$ we
have $\text{cd}(R\pi_*)\le\text{cd}(R\pi_{X*})<\infty.$

\begin{sublemma}\label{analytic-LO2-2.2.2}
Let $M$ be an AR-null complex in $\s D(\s A(\fk X)).$
Then $R\pi_*M=0.$
\end{sublemma}

\begin{proof}
The case when $M$ is essentially bounded below follows
from (\cite{Eke}, 1.1). In particular, since each of
$\s H^i(M)$ and $\tau_{>i}M$ is AR-null, we have $R\pi_*\s
H^i(M)\cong R\pi_*\tau_{>i}M=0.$ For the general case, we
apply (\cite{LO1}, 2.1.10), with $\epsilon=\pi$ and $\s
C_{\bullet}=$ the entire category of $\Lambda_{\bullet}$-modules on
$\fk X_{\text{lis-an}}^{\bb N};$ the conditions in \textit
{loc. cit.} are trivially satisfied, and the assumption
(\cite{LO1}, 2.1.7) for the ringed topoi $(\fk X_{\text{lis-an}},
\Lambda_n)$ is verified by (\cite{Dim}, 3.1.7, 3.4.1).
\end{proof}

For $M\in\s D(\s A),$ let $\widehat{M}$ be the \textit{normalization of
$M:$} $\widehat{M}=L\pi^*R\pi_*M.$ We say that $M$ is \textit{normalized} if
the coadjunction morphism $\widehat{M}\to M$ is an isomorphism. As
mentioned before, if $M\in\s D(\s A(X_{\text{an}}))$ for an analytic
space $X,$ one defines the normalization $\widehat{M}$ similarly, using
the morphism $\pi_X$ of analytic topoi. In this case, the normalization
functor has finite cohomological dimension.

The analytic versions of (\cite{LO2}, 2.2.1, 3.0.11, 3.0.10)
hold and can be proved verbatim, as we state in the following.

\begin{subproposition}\label{LO2-2.2.1}
\emph{(i)} For $(U,u)$ in $\emph{Lis-an}(\fk X)$ and $M\in\s
D(\fk X^{\bb N},\bb Z),$ we have $R\pi_{U*}(M_U)=(R\pi_*M)_U$ in $\s D(U_{\emph{an}},\bb Z).$

\emph{(ii)} For $(U,u)$ in $\emph{Lis-an}(\fk X)$ and $M\in\s
D(\fk X,\Lambda),$ we have $L\pi^*_U(M_U)=(L\pi^*M)_U$ in $\s
D(\s A(U_{\emph{an}})).$

\emph{(iii)} Let $M\in\s D(\s A).$ Then it is normalized if and only if the natural
morphism
$$
\Lambda_{n-1}\otimes^L_{\Lambda_n}M_n\to M_{n-1}
$$
is an isomorphism for each $n.$
\end{subproposition}
\end{anitem}

\begin{anitem}\label{fact1}
Let $f:X\to Y$ be a morphism of complex analytic spaces. Then
we have a natural isomorphism
$$
Rf_*\circ R\pi_{X*}\cong R\pi_{Y*}\circ Rf^{\bb N}_*:\s
D(X^{\bb N},\bb Z)\to\s D(Y,\bb Z).
$$
In fact, $f$ defines a morphism of their analytic topoi $f:X_{\text
{an}}\to Y_{\text{an}},$ and we have a commutative diagram of topoi:
$$
\xymatrix@C=.8cm @R=.6cm{
X_{\text{an}}^{\bb N} \ar[r]^-{f^{\bb N}} \ar[d]_-{\pi_X} &
Y_{\text{an}}^{\bb N} \ar[d]^-{\pi_Y} \\
X_{\text{an}} \ar[r]_-f & Y_{\text{an}}.}
$$
To see this, one verifies either $f^{\bb N,-1}\circ\pi_Y^{-1}\cong
\pi_X^{-1}\circ f^{-1}$ (which is clear) or $\pi_{Y*}\circ f^{\bb
N}_*\cong f_*\circ\pi_{X*}$ (namely, $f_*$ preserves limits).

One may generalize it to algebraic morphisms between algebraic analytic
stacks and their adic derived categories, using descent. As
we will not use it in the sequel, we do not give the proof
in detail here.
\end{anitem}

\begin{sublemma}\label{A''}
\emph{(i)} Let $X$ be a $\bb C$-scheme and $F\in\fk{Mod}_c(\Lambda_{X(\bb C)}).$
Then there is an integer $N\ge0$ such that $F/\emph{Ker}(\lambda^N)$
is a flat $\Lambda_{X(\bb C)}$-sheaf. Also, for each $n\ge0,$ the
sheaf $F\otimes\Lambda_n$ is constructible.

\emph{(ii)} Let $X$ be an analytic space and let $F$ be an
analytically constructible $\Lambda_X$-module. Then there is an
integer $N\ge0$ with the same property as above.
\end{sublemma}

\begin{proof}
(i) This is a consequence of (\cite{BBD}, 6.1.2, $(A'')$) and
(\cite{SGA4.5}, Rapport, 2.8). But since the sketchy proof of
(\cite{BBD}, 6.1.2, $(A'')$) is not very clear to me, we
prove (i) directly.

By definition there is a stratification $\s S$ of $X$ such
that for each stratum $i_V:V\hookrightarrow X,$ the sheaf
$i_V^*F$ on $V(\bb C)$ is a lcc $\Lambda_{V(\bb C)}$-module.
Thus there exists an integer $n_V\ge0$ such that $i_V^*F/\text{Ker}
(\lambda^{n_V}\text{ on }i_V^*F)$ is flat. To conclude, we take
$N=\max_V\{n_V\}$ and use the fact that $i_V^*$ is exact
(hence preserves $``\text{Ker}(\lambda^N)"$).

The second statement follows from (\ref{constr=serre}), as
$F\otimes\Lambda_n=\text{Coker}(\lambda^{n+1}\text{ on }F).$

(ii) The same proof applies; $\s S$ is now an analytic
stratification, with finitely many strata in it.
\end{proof}

Now we assume that our analytic stack $\fk X$ is the analytification
of an algebraic one $\mathcal X$ (except in (\ref{vanishing}), which
apply to non-algebraic ones too). Next we show that $R\pi_*$ and
$L\pi^*$ preserve constructibility. This depends in an essential
way on the lisse-analytic topology, as the corresponding statement
in the algebraic category is false (cf.\ (\cite{LO2}, 3.0.16)).

\begin{subproposition}\label{P-con}
The functors $R\pi_*$ and $L\pi^*$ restrict to functors
between $\s D_c(\s A(\fk X))$ and $\s D_c(\fk X,\Lambda).$
\end{subproposition}

\begin{proof}
1) We prove that if $M=(M_n)_n\in\s D_c(\s A(\fk X)),$ then
$R^i\pi_*M$ is a constructible $\Lambda_{\fk X}$-module for
each $i.$ Let us start with the following lemma, a little
stronger than needed.

\begin{sublemma}\label{limit-cart}
If $M\in\s D_{\emph{cart}}(\fk X^{\bb N},\bb Z),$ then
$R\pi_*M\in\s D_{\emph{cart}}(\fk X,\bb Z).$
\end{sublemma}

\begin{proof}
Let $f:U\to V$ be a morphism in $\text{Lis-an}(\fk X).$ It
induces a commutative diagram of topoi
$$
\xymatrix@C=.8cm @R=.6cm{
U_{\text{an}}^{\bb N} \ar[r]^-{f^{\bb N}} \ar[d]_-{\pi_U} &
V_{\text{an}}^{\bb N} \ar[d]^-{\pi_V} \\
U_{\text{an}} \ar[r]_-f & V_{\text{an}}.}
$$
In particular, we have the base change morphism
$$
bc:f^*R\pi_{V*}M_V\to R\pi_{U*}f^{\bb N,*}M_V.
$$
The transition morphism $\Theta_f:f^*(R\pi_*M)_V\to(R\pi_*M)_U$
is the composition (via (\ref{LO2-2.2.1} i))
$$
\xymatrix@C=1.2cm{
f^*R\pi_{V*}M_V \ar[r]^-{bc} & R\pi_{U*}f^{\bb N,*}M_V
\ar[r]^-{R\pi_{U*}(\theta_f)} & R\pi_{U*}M_U,}
$$
where $\theta_f$ is the transition morphism for $(M_n)$
and is an isomorphism (\ref{cart-complex}). We need to
show that the base change morhpism is an isomorphism.

Since $R\pi_{U*}$ and $R\pi_{V*}$ have finite
cohomological dimensions, by taking $\s H^i$ of both sides
of the base change morphism, we may assume that $M$ is essentially
bounded (i.e.\ the projective system
$\s H^i(M)$ is AR-null for $|i|\gg0$), or even a projective system of
Cartesian sheaves (\ref{constr=serre}).

A standard fiber-product argument reduces us to assume
that $f$ is smooth (cf.\ (\cite{LMB}, 12.3.1)). Since the problem
is local on $U,$ we may assume that $f$ is isomorphic to the
projection $\text{pr}_1:V\times Z\to V$ from a product
of $V$ with a complex manifold $Z.$

Let us denote by $f^*_{ps}$ the presheaf inverse image functor,
and by $\s R^i\pi_{U*}M_U$ the presheaf that assigns to an
open $U'\subset U$ the group $H^i(\pi_U^*U',M_U),$ and
similarly for $\s R^i\pi_{V*}M_V.$ Then we have morphisms of
presheaves on $U:$
$$
f^*_{ps}\s R^i\pi_{V*}M_V\to\s R^i\pi_{U*}f_{ps}^{\bb N,*}M_V
\to\s R^i\pi_{U*}f^{\bb N,*}M_V
$$
from which the base change morphism for cohomology sheaves
$R^i\pi_{V*}M_V$ is deduced by sheafification. Explicitly,
to an open $U'\subset U,$ the left-hand side
assigns the colimit of the direct system $H^i(\pi_V^*V',M_V)$
parametrized by open sets $V'\subset V$ containing $f(U'),$
and the right-hand side assigns $H^i(\pi_U^*U',f^{\bb
N,*}M_V),$ and the morphism of presheaves is induced from
the natural maps
$$
H^i(V',M_{n,V})\to H^i(U',f^*M_{n,V})
$$
given by the mapping $f|_{U'}:U'\to V'.$ To show that this
morphism of presheaves sheafifies to an isomorphism, it
suffices (by checking stalks) to show that it is an isomorphism
on a topological basis on $U,$ which can be taken to be open
sets of the form $U'=V'\times Z',$ with $V'\subset V$ and
$Z'\subset Z$ open, and the $Z'$'s are polydisks. Then
$V'=f(U')$ is open, and by the K\"unneth formula, the morphism
$$
H^i(V',M_{n,V})\to H^i(U',M_{n,V}\boxtimes\bb Z)
$$
is an isomorphism, since $R\Gamma(Z',\bb Z)=\bb Z.$
\end{proof}

Now let $M\in\s D_c(\s A(\fk X)).$ Since $R\pi_*M\in\s
D_{\text{cart}}(\fk X,\Lambda),$ to show that it is constructible,
by (\ref{L-analytic-constr}, \ref{LO2-2.2.1} i) we may pass
to an algebraic presentation $X(\bb C)\to\fk X,$ so we may assume that
$\fk X=X(\bb C)$ for some $\bb C$-scheme $X.$ Since $R\pi_*$
has finite cohomological dimension, we may assume that $M$ is an
AR-adic projective system with constructible components. By
(\cite{SGA5}, V, 3.2.3) we may assume that $M$ is an adic system.

Next we will reduce to the case where $M$ has lcc components.
Applying (\cite{BBD}, 6.1.2 $(A')$) to the components
$M_n$ and the transition maps $\rho_n:M_n\to M_{n-1},$ we
see that $M$ is algebraic, i.e.\ it is the analytification of
an adic system $\widetilde{M}=(\widetilde{M}_n,\widetilde{\rho}_n)$ with
constructible components on $X_{\text{\'et}}.$ By (\cite{SGA4.5},
Rapport, 2.5), there exists a stratification of $X$ over each
stratum of which $\widetilde{M}$ is lisse. Let $j:U\hookrightarrow
X$ be an open stratum, with complement $i:Z\to X.$ Then
$M_{n,U(\bb C)},$ being the analytification of the lcc sheaf
$\widetilde{M}_{n,U_{\text{\'et}}},$ is lcc; it corresponds to the representation
$$
\xymatrix@C=.5cm{
\pi_1(U(\bb C)) \ar[r] & \pi_1^{\text{\'et}}(U) \ar[r] &
\text{GL}(\widetilde{M}_{n,U_{\text{\'et}},\overline{u}}),}
$$
and $U(\bb C)$ is locally contractible (\cite{BV}, 4.4).
We apply $R\pi_*$ to the exact triangle
$$
\xymatrix@C=.5cm{
i^{\bb N}_*Ri^{\bb N,!}M \ar[r] & M \ar[r] & Rj^{\bb
N}_*M_{U(\bb C)} \ar[r] &}
$$
to obtain (by (\ref{fact1}))
$$
\xymatrix@C=.5cm{
i_*R\pi_{Z(\bb C)*}Ri^{\bb N,!}M \ar[r] & R\pi_*M \ar[r] &
Rj_*R\pi_{U(\bb C)*}M_{U(\bb C)} \ar[r] &.}
$$
We have $Ri^{\bb N,!}M\in\s D_c(\s A(Z(\bb C))),$ so as noetherian
induction hypothesis, we may assume that $R\pi_{Z(\bb C)*}Ri^{\bb N,!}M\in\s
D_c(Z(\bb C),\Lambda),$ and hence $i_*R\pi_{Z(\bb C)*}Ri^{\bb N,!}M\in\s
D_c(X(\bb C),\Lambda).$ Recall that $M$ is assumed to be a system of
sheaves and $R\pi_{U(\bb C)*}$ has finite cohomological dimension,
therefore $R\pi_{U(\bb C)*}M_{U(\bb C)}$ is bounded, and since
$Rj_*$ preserves constructibility (\cite{Dim}, 4.1.5), we
reduce to the case where all $M_n$'s are lcc.

\begin{sublemma}\label{vanishing}
Under the assumption that the $M_n$'s are lcc on $\fk X,$ we
have $R^i\pi_*M=0$ for $i\ne0.$
\end{sublemma}

\begin{proof}
By (\ref{LO2-2.2.1} i) we may pass to an analytic presentation
$X\to\fk X,$ so assume that $\fk X=X$ is an analytic space. Then
$R^i\pi_*M$ is the sheaf on $X_{\text{an}}$ associated to the
presheaf that assigns to each open set $U\subset X$ the group
$H^i(\pi^*U,M).$ We only need to consider those open sets $U$
which are contractible, since they generate a basis
for the topology of $X,$ by (\cite{BV}, 4.4). Then $M_n|_U,$
a priori locally constant, are constant sheaves defined by
finite sets, hence $H^i(U,M_n)=0$ for $i\ne0,$ and $H^0(U,M_n)$
are finite so that $\varprojlim_n^1H^0(U,M_n)=0.$ The result then
follows from (\cite{LO2}, 2.1.i).
\end{proof}

It remains to show that $\pi_*M$ is constructible;
in fact it is a lcc $\Lambda_{X(\bb C)}$-module. To see
this, we may replace $X(\bb C)$ by contractible open subsets
by (\ref{analytic-9.1}), hence assume that each $M_n$
is constant. Then $\pi_*M=\varprojlim_nM_n$ is also
constant, because this sheaf limit is just the presheaf limit.

\vskip.3truecm

2) Now we prove that if $F\in\s D_c(\fk X,\Lambda),$ then
$L\pi^*F\in\s D_c(\s A).$

First let $F\in\s D_{\text{cart}}(\fk X,\Lambda),$ and let us
show that $L\pi^*F$ is Cartesian, i.e.\ for all $i,$ each component of
the projective system $L^i\pi^*F$ is Cartesian. Since $L\pi^*$ has
finite cohomological dimension we may assume that $F$ is a Cartesian
sheaf, by (\ref{constr=serre}). Then
$(L\pi^*F)_n=F\otimes^L_{\Lambda}\Lambda_n$ is represented by the
complex $F\overset{\lambda^{n+1}}{\longrightarrow}F,$ both the kernel
and cokernel of which are Cartesian sheaves, by (\ref{constr=serre}).

Now let $F\in\s D_c(\fk X,\Lambda),$ and let us show that $L\pi^*F$
is a $\lambda$-complex, i.e.\ each cohomology $L^i\pi^*F$ is
an AR-adic system of constructible sheaves. By
(\ref{L-analytic-constr}, \ref{LO2-2.2.1}), this can be checked on
an algebraic presentation $X(\bb C),$ so we assume that $\fk X=X(\bb C)$
and $F\in\fk{Mod}_c(\Lambda_{X(\bb C)}).$ By (\ref{A''} i)
we reduce to two cases: $F$
is flat, or $F$ is annihilated by $\lambda.$

If $F$ is flat, then $L\pi^*F=\pi^*F$ is the adic sheaf
$(F\otimes\Lambda_n)_n,$ with constructible components (\ref{A''} i)
(even with respect to the same algebraic stratification
for $F$).

If $\lambda F=0,$ then using the following projective system
of $\Lambda$-flat resolutions of the $\Lambda_n$'s
$$
\xymatrix@C=.9cm @R=.5cm{
0 \ar[r] & \Lambda \ar[r]^-{\lambda^{n+1}} \ar[d]_-{\lambda} &
\Lambda \ar[r] \ar[d]_-1 & \Lambda_n \ar[r] \ar[d] & 0 \\
0 \ar[r] & \Lambda \ar[r]^-{\lambda^n} & \Lambda \ar[r] & \Lambda_{n-1}
\ar[r] & 0}
$$
we see that $L\pi^*F$ is represented by the following complex of systems:
$$
\xymatrix@C=.8cm @R=.5cm{
0 \ar[r] & F \ar[r]^-0 \ar[d]^-0 & F \ar[d]^-1 \ar[r] & 0 \\
0 \ar[r] & F \ar[r]^-0 & F \ar[r] & 0.}
$$
Therefore, $\s H^0(L\pi^*F)$ is
$$
\xymatrix@C=.5cm{
\cdots \ar[r]^-1 & F \ar[r]^-1 & F \ar[r]^-1 & \cdots}
$$
which is adic with constructible components, and $\s H^{-1}(L\pi^*F)$ is
$$
\xymatrix@C=.5cm{
\cdots \ar[r]^-0 & F \ar[r]^-0 & F \ar[r]^-0 & \cdots}
$$
which is AR-null (hence AR-adic) with constructible components.
\end{proof}

By (\ref{analytic-LO2-2.2.2}) the functor $R\pi_*:\s D_c(\s A(\fk
X))\to\s D_c(\fk X,\Lambda)$
factors through the quotient category $D_c(\fk X,\Lambda):$
$$
\xymatrix@C=1cm{
\s D_c(\s A) \ar[r]^-Q & D_c(\fk X,\Lambda)
\ar[r]^-{\overline{R\pi}_*} & \s D_c(\fk X,\Lambda).}
$$
Thus the normalization functor, when restricted to $\s D_c(\s A),$
factors through $D_c(\fk X,\Lambda),$ and for $K\in D_c(\fk
X,\Lambda)$ we still denote $L\pi^*\overline{R\pi}_*K$ by $\widehat{K}.$

\begin{subproposition}\label{agree}
\emph{(i)} The functors $(Q\circ L\pi^*,\overline{R\pi}_*)$
induce an equivalence $D_c(\fk X,\Lambda)
\longleftrightarrow\s D_c(\fk X,\Lambda).$

\emph{(ii)} Let $f:\mathcal X\to\mathcal Y$ be a morphism of
complex algebraic stacks, and let $f^{\emph{an}}:
\mathfrak X\to\mathfrak Y$ be its analytification. Then
the following diagram commutes:
$$
\xymatrix@C=1cm @R=.8cm{
D_c^+(\mathfrak X,\Lambda) \ar[r]^-{\overline{R\pi}
_{\mathcal X,*}} \ar[d]_-{f_*^{\emph{an}}} & \s
D_c^+(\mathfrak X,\Lambda) \ar[d]^-{f_*^{\emph{an}}} \\
D_c^+(\mathfrak Y,\Lambda) \ar[r]^-{\overline{R\pi}
_{\mathcal Y,*}} & \s D_c^+(\mathfrak Y,\Lambda).}
$$
\end{subproposition}

\begin{proof}
(i) We will show that the adjunction and coadjunction
maps are isomorphisms. For coadjunction maps, this is the
analytic version of (\cite{LO2}, 3.0.14).

\begin{sublemma}\label{analytic-LO2-3.0.14}
Let $M\in\s D_c(\s A(\mathfrak X)).$ Then the coadjunction map
$\widehat{M}\to M$ has an AR-null cone.
\end{sublemma}

\begin{proof}
It can be proved in the same way as (\cite{LO2}, 3.0.14).
We go over the proof briefly. First note that, if
$$
\xymatrix@C=.5cm{
M' \ar[r] & M \ar[r] & M'' \ar[r] & M'[1]}
$$
is an exact triangle in $\s D_c(\s A)$ and the coadjunction map
is an isomorphism for two vertices, then it is so for the third.
In particular, by (\ref{analytic-LO2-2.2.2}), if $M$ and $M'$ in
$\s D_c(\s A)$ are AR-isomorphic (that is, their images in $D_c(\fk X,\Lambda)$ are isomorphic), then the coadjunction map is an
isomorphism for $M$ if and only if it is so for $M'.$

By (\ref{LO2-2.2.1}), we may pass to an algebraic presentation
$P:X(\bb C)\to\fk X,$ so assume that $\fk X=X(\bb C).$ The normalization
functor has finite cohomological dimension,
so one can assume that $M$ is a $\lambda$-module.
By (\cite{SGA5}, V, 3.2.3)
we may assume that $M$ is an adic system with constructible
components. Therefore, $M$ is algebraic by
(\cite{BBD}, 6.1.2 $(A')$), and by (\cite{SGA4.5}, Rapport,
2.8) one reduces to two cases: $M$ is flat (i.e.\ each component
$M_n$ is a flat $\Lambda_n$-sheaf), or $\lambda M=0,$ i.e.\
$M$ is AR-isomorphic to (hence we may assume that it is) the constant system
$(M_0)_n.$

If $M$ is flat, then the natural map
$$
M_n\otimes^L_{\Lambda_n}\Lambda_{n-1}\simeq M_n\otimes
_{\Lambda_n}\Lambda_{n-1}\overset{\sim}{\to}M_{n-1}
$$
is an isomorphism, so by (\ref{LO2-2.2.1} iii) $M$ is normalized,
hence the cone of $\widehat{M}\to M$ is zero.

If $M$ is the constant adic system $(M_0)_n,$ then
$R\pi_*M=M_0$ by (\cite{LO2}, 2.2.3). We saw in the proof
of (\ref{P-con}) that $\s H^0(L\pi^*M_0)$ is $(M_0)_n$ and that
$\s H^{-1}(L\pi^*M_0)$ is AR-null. Therefore, the natural
map $L\pi^*M_0\to M$ is an AR-isomorphism.
\end{proof}

Then we prove the following, slightly general than needed.

\begin{sublemma}\label{gg}
Let $\fk X$ be an analytic stack (not necessarily algebraic),
and let $F\in\s D_c(\fk X,\Lambda)$ be a complex with
analytically constructible cohomology sheaves. Then the
adjunction map $F\to R\pi_*L\pi^*F$ is an isomorphism.
\end{sublemma}

\begin{proof}
Let us denote $R\pi_*L\pi^*F$ by
$\check{F}.$ Note that if $F'\to F\to F''\to F'[1]$ is an
exact triangle, and the adjunction map is an isomorphism
for two vertices, then it is so for the third.

That the map $F\to\check{F}$ is an isomorphism is a
local property, since it is equivalent to the vanishing
of all the cohomology sheaves of the cone, which can be
checked locally. So by (\ref{LO2-2.2.1} ii), we may replace
$\fk X$ by an analytic presentation $X.$
Since the functor $F\mapsto\check{F}$ has finite
cohomological dimension, we may assume that $F$ is a sheaf.
By (\ref{A''} ii) we reduce to two cases: $F$ is flat, or
$F$ is annihilated by $\lambda.$ The second case follows from
(\cite{LO2}, 2.2.3), so we assume that $F$ is flat.

We want to reduce to the case
when $F$ is locally constant. Let $j:U\hookrightarrow
X$ be the open immersion of a subspace over
which $F$ is locally constant, and let $i:Z\hookrightarrow
X$ be the complement. Consider the exact triangle
$$
\xymatrix@C=.6cm{
i_*N \ar[r] & F \ar[r] & Rj_*F_U \ar[r] &,}
$$
where $N=Ri^!F$ is in $\s D^b_c(Z,\Lambda)$ by (\cite{Dim}, 4.1.5 i).
It suffices to show that the adjunction maps for $i_*N$ and
$Rj_*F_U$ are isomorphisms.

By (\ref{fact1}) we have $R\pi_{X,*}\circ i_*^{\bb N}\simeq i_*\circ
R\pi_{Z,*}.$ Also we have $L\pi^*_X\circ i_*\simeq i_*^{\bb
N}\circ L\pi^*_Z,$ since $i_*$ is extension by zero (that $i_*(N\otimes^L_{\Lambda}\Lambda_n)\simeq
i_*N\otimes^L_{\Lambda}\Lambda_n$ also follows from the
projection formula in topology (\cite{Dim}, 2.3.29)). Therefore, the
adjunction map for $i_*N$ on $X$ is obtained by applying
$i_*$ to the adjunction map for $N$ on $Z:$
$$
i_*N\to R\pi_{X,*}L\pi_X^*i_*N\simeq i_*R\pi_{Z,*}L\pi_Z^*N,
$$
which is an isomorphism by noetherian hypothesis.

Again by (\ref{fact1}) we have $R\pi_{X,*}\circ Rj_*^{\bb N}\simeq
Rj_*\circ R\pi_{U,*}.$ We will show that $Rj_*^{\bb N}\circ L\pi_U^*\simeq
L\pi_X^*\circ Rj_*$ on $\s D_c(U,\Lambda),$ even though we
only need this isomorphism for bounded complexes.
Let $F\in\s D_c(U,\Lambda).$ For each $n$ we
have a natural morphism $\Lambda_n\otimes^L_{\Lambda}Rj_*F\to
Rj_*(\Lambda_n\otimes^L_{\Lambda}F).$ Consider the short
exact sequence
$$
\xymatrix@C=.7cm{
0 \ar[r] & \Lambda \ar[r]^-{\lambda^{n+1}} & \Lambda \ar[r] &
\Lambda_n \ar[r] & 0.}
$$
Let $F\to I$ be a $K$-injective resolution of $F$ (cf.\ \cite{Spa}).
Then $\Lambda_n\otimes^L_{\Lambda}I$ is also a $K$-injective
complex (one sees
this by applying (\cite{Spa}, 1.3) to the exact triangle
$$
\xymatrix@C=.7cm{
I \ar[r]^-{\lambda^{n+1}} & I \ar[r] & \Lambda_n\otimes^LI \ar[r] &}
$$
obtained from the short exact sequence above tensored with $I$), and
$j_*(\Lambda_n\otimes^LI)=\Lambda_n\otimes^Lj_*I,$ since by
applying $Rj_*$
to the exact triangle of $K$-injective complexes
$$
\xymatrix@C=.7cm{
I \ar[r]^-{\lambda^{n+1}} & I \ar[r] & \Lambda_n\otimes^LI
\ar[r] &}
$$
we get
$$
\xymatrix@C=.7cm{
j_*I \ar[r]^-{\lambda^{n+1}} & j_*I \ar[r] &
j_*(\Lambda_n\otimes^LI) \ar[r] &},
$$
and by applying $-\otimes^Lj_*I$ to the short exact sequence
above we get
$$
\xymatrix@C=.7cm{
j_*I \ar[r]^-{\lambda^{n+1}} & j_*I \ar[r] &
\Lambda_n\otimes^Lj_*I \ar[r] &.}
$$
Thus by (\cite{Spa}, 5.12, 6.7) we have
$$
Rj_*(\Lambda_n\otimes^LF)=j_*(\Lambda_n\otimes^LI)
=\Lambda_n\otimes^LRj_*F.
$$

Therefore, the adjunction map for $Rj_*F_U$ on
$X$ is obtained by applying $Rj_*$ to the adjunction
map for $F_U$ on $U.$ Hence we may assume that
$F$ is a locally constant sheaf on $X.$ Replacing $X$ by
an open cover, we assume that $F$ is constant, defined
by a free module $\Lambda^r.$ By additivity we may assume that $r=1.$
Then $L\pi^*\Lambda=\Lambda_{\bullet},$ and
$\pi_*\Lambda_{\bullet}=\varprojlim\Lambda_{\bullet}=\Lambda.$
We conclude by applying (\ref{vanishing}) to deduce that
$R^i\pi_*\Lambda_{\bullet}=0$ for $i\ne0.$
\end{proof}

Therefore, $(Q\circ L\pi^*,\overline{R\pi}_*)$ induce an
equivalence between $D_c(\fk X,\Lambda)$ and
$\s D_c(\fk X,\Lambda).$

(ii) If $X_{\bullet}\to\fk X$ is a strictly
simplicial algebraic smooth hypercover, we have
$D_c(\fk X,\Lambda)\simeq D_c(X_{\bullet},\Lambda)$
and $\s D_c(\fk X,\Lambda)\simeq\s D_c(X_{\bullet},\Lambda)$
by (\ref{P-compar} i).
So we may assume that $\fk X=X$ and $\fk Y=Y$ are
analytifications of algebraic schemes. By definition of
$\overline{R\pi}_*,$ it suffices to show that the following
diagram commutes
$$
\xymatrix@C=1cm @R=.7cm {
\s D_c^+(\s A(X)) \ar[r]^-{R\pi_{X,*}} \ar[d]_-{f_*^{\text{an},\bb
N}} & \s D_c^+(X,\Lambda) \ar[d]^-{f_*^{\text{an}}} \\
\s D_c^+(\s A(Y)) \ar[r]^-{R\pi_{Y,*}} & \s D_c^+(Y,\Lambda),}
$$
and this follows from the commutativity of the diagram of topoi
$$
\xymatrix@C=.8cm @R=.6cm{
X_{\text{an}}^{\bb N} \ar[r]^-{\pi_X}
\ar[d]_-{f^{\text{an},\bb N}} & X_{\text{an}} \ar[d]^-{f^{\text{an}}} \\
Y_{\text{an}}^{\bb N} \ar[r]^-{\pi_Y} & Y_{\text{an}}.}
$$
Note that the corresponding diagram for $f^{\text{an}}:\fk
X\to\fk Y$ does not even make sense, since if $f$ is
not smooth, it does not necessarily induce a morphism of
their lisse-analytic topoi.
\end{proof}

\begin{subremark}\label{R-rat}
Similarly, $\overline{R\pi}_{\mathcal X,*}$ induces a fully
faithful functor $D_c(\fk X,\overline{\bb Q}_{\ell})\to\s D_c(\fk
X,\overline{\bb Q}_{\ell}),$ which is compatible with
$f^{\text{an}}_*$ when restricted to $D_c^+.$
\end{subremark}

\begin{subremark}
This result (\ref{agree}), together with (\ref{compar-cohom})
and (\cite{LO2}, 3.1.6), generalizes (\cite{BBD}, 6.1.2, $(B'')$).
Taking their cores, we obtain
$$
\Lambda\text{-Sh}_c(\mathcal X)\simeq\Lambda\text{-Sh}_c(\fk
X)\simeq\fk{Mod}_c(\Lambda_{\fk X}),
$$
generalizing (\textit{loc. cit.}, $(A'')$).
\end{subremark}

\section{Decomposition Theorem over $\bb C$}\label{sec-decomp-C}

Let $(\Lambda,\mathfrak m)$ be a complete DVR as before,
with residue characteristic $\ell\ne2.$ Let $\mathcal X$
be an algebraic stack over $\text{Spec }\bb C.$ We
first prove a comparison theorem between the
lisse-\'etale topoi over $\bb C$ and over $\bb
F,$ and then use this together with (\ref{compar-cohom},
\ref{agree}) to deduce the decomposition theorem for $\bb
C$-algebraic stacks with affine stabilizers.

\subsection{Comparison between the lisse-\'etale topoi
over $\bb C$ and over $\bb F$}

Let $(\s S,\mathcal L)$ be a pair on $\mathcal X$
with $\Lambda_0$-coefficients. By
refining we may assume that all strata in $\s S$ are
essentially smooth and connected. Let $A\subset\bb C$ be a
subring of finite type over $\bb Z,$ large enough so
that there exists a triple $(\mathcal X_S,\s
S_S,\mathcal L_S)$ over $S:=\text{Spec }A$ giving rise to
$(\mathcal X,\s S,\mathcal L)$ by base change, that $\mathcal X_S$ is flat over $S,$ and that
$1/\ell\in A.$ Then $S$ satisfies the condition (LO); the
hypothesis on $\ell$-cohomological dimension follows from
(\cite{SGA4}, X, 6.2). We may shrink $S$ to assume that
strata in $\s S_S$ are smooth over $S$ with
geometrically connected fibers, which is possible because
one can take a presentation $P:X_S\to\mathcal X_S$ and
shrink $S$ so that the strata in $P^*\s S_S$ are
smooth over $S$ with geometrically connected fibers. Let
$a:\mathcal X_S\to S$ be the structural map.

Let $A\subset V\subset\bb C,$ where $V$ is a strict henselian discrete valuation ring
whose residue field $s$ is an algebraic closure of a finite residue field
of $A.$ Let $(\mathcal X_V,\s S_V,\mathcal L_V)$ be the triple
over $V$ obtained by base change, and let $(\mathcal X_s,\s S_s,\mathcal
L_s)$ be its special fiber. Then we have morphisms
$$
\xymatrix@C=1cm{
\mathcal X \ar[r]^-u & \mathcal X_V & \mathcal X_s
\ar[l]_-i.}
$$

\begin{subproposition}\label{6.1.9}\emph{(stack version of
(\cite{BBD}, 6.1.9))}
For $S$ small enough, the functors
$$
\xymatrix@C=1cm{
D^b_{\s S,\mathcal L}(\mathcal X,\Lambda_n) &
D^b_{\s S_V,\mathcal L_V}(\mathcal X_V,\Lambda_n)
\ar[l]_-{u^*_n} \ar[r]^-{i^*_n} & D^b_{\s
S_s,\mathcal L_s}(\mathcal X_s,\Lambda_n)}
$$
and
$$
\xymatrix@C=1cm{
D_{\s S,\mathcal L}^b(\mathcal X,\Lambda) & D_{\s S_V,\mathcal
L_V}^b(\mathcal X_V,\Lambda) \ar[l]_-{u^*} \ar[r]^-{i^*} & D_{\s
S_s,\mathcal L_s}^b(\mathcal X_s,\Lambda)}
$$
are equivalences of triangulated categories with standard
$t$-structures.
\end{subproposition}

\begin{proof}
These restriction functors are clearly triangulated
functors preserving the standard $t$-structures.

By (\ref{Tbc}), we can shrink $S=\text{Spec }A$ so that
for any $F$ and $G$ of the form $j_!L,$ where $j:\mathcal
U_S\to\mathcal X_S$ in $\s S_S$ and $L\in\mathcal
L_S(\mathcal U_S),$ the formations of $R\s Hom
_{\mathcal X_S}(F,G)$ commute with base change on $S,$
and that the complexes $a_*\s Ext^q_{\mathcal X_S}(F,G)$
on $S$ are lcc (see Remark \ref{referee21} below for explanation) and of formation compatible with base
change, i.e.\ the cohomology sheaves are lcc, and for any
$g:S'\to S,$ the base change morphism for $a_*:$
$$
g^*a_*\s Ext^q_{\mathcal X_S}(F,G)\to a_{S'*}g'^*
\s Ext^q_{\mathcal X_S}(F,G)
$$
is an isomorphism. Then using the same argument as in
\cite{BBD}, the claim for $u^*_n$ and $i^*_n$ follows.
For the reader's convenience, we explain the proof in
\cite{BBD} in more detail.

Note that the spectra of $V,\ \bb C$ and $s$ have no
non-trivial \'etale surjections mapping to them, so their
small \'etale topoi are equivalent to the topos of sets. In particular,
$Ra_{V*}$ (resp.\ $Ra_{\bb C*}$ and $Ra_{s*}$) is just
$R\Gamma.$ Let us show the full faithfulness of $u^*_n$
and $i^*_n$ first. For $K,L\in D^b_{\s S_V,\mathcal
L_V}(\mathcal X_V,\Lambda_n),$ let $K_{\bb C}$ and
$L_{\bb C}$ (resp.\ $K_s$ and $L_s$) be their images
under $u^*_n$ (resp.\ $i^*_n$). Then the full faithfulness
follows from the more general claim that, the maps
$$
\xymatrix@C=.6cm{
Ext^i_{\mathcal X}(K_{\bb C},L_{\bb C}) &
Ext^i_{\mathcal X_V}(K,L) \ar[l]_-{u^*_n} \ar[r]^-{i^*_n}
& Ext^i_{\mathcal X_s}(K_s,L_s)}
$$
are bijective for all $i.$

Since $Hom_{D_c(\mathcal X,\Lambda_n)}(K,-)$ and
$Hom_{D_c(\mathcal X,\Lambda_n)} (-,L)$ are cohomological
functors, by 5-lemma we may assume that $K=F$ and $L=G$
are $\Lambda_n$-sheaves. Let $j:\mathcal U_S\to\mathcal
X_S$ be the immersion of an open stratum in $\s
S_S,$ with complement $i:\mathcal Z_S\to\mathcal X_S.$
Using the short exact sequence
$$
\xymatrix@C=.5cm{
0 \ar[r] & j_{V!}j_V^*F \ar[r] & F \ar[r] & i_{V*}i_V^*F
\ar[r] & 0}
$$
and noetherian induction on the support of $F$ and $G,$
we may assume that they take the form $j_{V!}L,$ where
$j$ is the immersion of some stratum in $\s S_S,$
and $L$ is a sheaf in $\mathcal L_V.$ The spectral
sequence
$$
R^pa_{\Box,*}\s Ext^q_{\mathcal X_{\Box}}(F_{\Box},
G_{\Box})\Longrightarrow Ext^{p+q}_{\mathcal X_{\Box}}
(F_{\Box},G_{\Box})
$$
is natural in the base $\Box,$ which can be $V,\ \bb C$
or $s.$ The assumption on $S$ made before implies that
the composite base change morphism
$$
g^*a_*\s Ext^q_{\mathcal X_S}(F,G)\to a_{S'*}g'^*
\s Ext^q_{\mathcal X_S}(F,G)\to a_{S'*}\s
Ext^q_{\mathcal X_{S'}}(g'^*F,g'^*G)
$$
is an isomorphism, for all $g:S'\to S.$ Therefore, the
maps
$$
\xymatrix@C=.6cm{
Ext^i_{\mathcal X}(F_{\bb C},G_{\bb C}) &
Ext^i_{\mathcal X_V}(F,G) \ar[l]_-{u^*_n} \ar[r]^-
{i^*_n} & Ext^i_{\mathcal X_s}(F_s,G_s)}
$$
are bijective for all $i.$ The claim (hence the full
faithfulness of $u^*_n$ and $i^*_n$) follows.

This claim also implies their essential surjectivity. To
see this, let us give a lemma first.

\begin{sublemma}\label{L5.2}
Let $F:\s C\to\s D$ be a triangulated functor
between triangulated categories. Let $A,B\in\emph{Obj
}\s C,$ and let $F(A)\overset{v}{\to}F(B)\to C'\to
F(A)[1]$ be an exact triangle in $\s D.$ If the map
$$
F:Hom_{\s C}(A,B)\to Hom_{\s D}(F(A),F(B))
$$
is surjective, then $C'$ is in the essential image of $F.$
\end{sublemma}

\begin{proof}
Let $u:A\to B$ be a morphism such that $F(u)=v.$ Let $C$
be the mapping cone of $u,$ i.e.\ let the triangle
$A\overset{u}{\to}B\to C\to A[1]$ be exact. Then its
image
$$
\xymatrix@C=.6cm{
F(A) \ar[r]^v & F(B) \ar[r] & F(C) \ar[r] & F(A)[1]}
$$
is also an exact triangle. This implies that $C'\simeq
F(C).$
\end{proof}

Now we can show the essential surjectivity of $u^*_n$ and
$i^*_n.$ For $K\in D^b_{\s S,\mathcal
L}(\mathcal X,\Lambda_n),$ to show that $K$
lies in the essential image of $u^*_n,$ using the
truncation exact triangles and (\ref{L5.2}), we reduce to
the case where $K$ is a sheaf. Using noetherian induction
on the support of $K,$ we reduce to the case where
$K=j_!L,$ where $j:\mathcal U\to\mathcal X$ is the
immersion of a stratum in $\s S,$ and
$L\in\mathcal{L(U)}.$ This is in the essential image of $u_n^*.$
Similarly, $i^*_n$ is also essentially surjective.

\vskip.3truecm

Next, we prove that $u^*$ and $i^*$ are equivalences.

We claim that for $K,L\in D_c^b(\mathcal X_V,\Lambda),$
if the morphisms
$$
\xymatrix@C=.7cm @R=.1cm{
& Hom_{D_c(\mathcal X,\Lambda_n)}(\widehat{K}_{n,\bb
C},\widehat{L}_{n,\bb C}) \\
Hom_{D_c(\mathcal X_V,\Lambda_n)}(\widehat{K}_n,
\widehat{L}_n) \ar[ur]^-{u_n^*} \ar[dr]_-{i_n^*} & \\
& Hom_{D_c(\mathcal X_s,\Lambda_n)}(
\widehat{K}_{n,s},\widehat{L}_{n,s})}
$$
are bijective for all $n,$ then the morphisms
$$
\xymatrix@C=.7cm @R=.1cm{
& Hom_{D_c(\mathcal X,\Lambda)}(K_{\bb C},L_{\bb C}) \\
Hom_{D_c(\mathcal X_V,\Lambda)}(K,L) \ar[ur]^-{u^*}
\ar[dr]_-{i^*} & \\
& Hom_{D_c(\mathcal X_s,\Lambda)}(K_s,L_s)}
$$
are bijective. Let $\Box$ be one of the bases $V,\ \bb
C$ or $s.$ Since $K$ and $L$ are bounded, we see from
the spectral sequence
$$
R^pa_{\Box,*}\s Ext^q_{\mathcal X_{\Box}}
(\widehat{K}_{n,\Box},\widehat{L}_{n,\Box})
\Longrightarrow Ext^{p+q}_{\mathcal X_{\Box}}
(\widehat{K}_{n,\Box},\widehat{L}_{n,\Box})
$$
and the finiteness of $R\s Hom$ and $Ra_{\Box,*}$
(\cite{LO1}, 4.2.2, 4.1) that, the groups
$Ext^{-1}(\widehat{K}_{n,\Box},\widehat{L}_{n,\Box})$ are
finite for all $n,$ hence they form a projective system
satisfying the condition (ML) (cf.\ EGA $0_{\text{III}},$ 13.1.2). By (\cite{LO2}, 3.1.3), we
have an isomorphism
$$
Hom_{D_c(\mathcal X_{\Box},\Lambda)}(K_{\Box},L_{\Box})
\overset{\sim}{\to}\varprojlim_nHom_{D_c(\mathcal
X_{\Box},\Lambda_n)}
(\widehat{K}_{n,\Box},\widehat{L}_{n,\Box}),
$$
natural in the base $\Box,$ and the claim follows.

Since when restricted to $D_{\s S_{\Box},\mathcal
L_{\Box}}^b,$ the functors $u^*_n$ and $i^*_n$ are fully
faithful for all $n,$ we deduce that $u^*$ and $i^*$ are
also fully faithful.

Finally we prove the essential surjectivity of $u^*$ and
$i^*.$ Let $K\in D^b_{\s S,\mathcal L}(\mathcal X,\Lambda).$
By the full faithfulness of $u^*$ and (\ref{L5.2}), we may
assume that $K$ is in the core $(D^b_{\s S,\mathcal L})^{\heartsuit}$ of $D^b_{\s S,\mathcal L}$ with
respect to the standard $t$-structure. Then there exists an
AR-adic representative $M=\{M_n,\rho_n:M_n\to M_{n-1}\}$ of
$K$ in $\s A(\mathcal X)$ that
is trivialized by $(\s S,\mathcal L);$ for instance $M=\s
H^0(\widehat{K})$ by (\cite{Sun}, 3.5). By
(\cite{SGA5}, V, 3.2.3), since $M$ is AR-adic, it
satisfies the condition (MLAR) (see (\cite{SGA5}, V, 2.1.1) for definition) and, if we denote by
$N=(N_n)_n$ the projective system of the universal images of
$M,$ there exists an integer $k\ge0$ such that
$l_k(N):=(N_{n+k}\otimes\Lambda_n)_n$ is an adic system. By
construction, the system $l_k(N)$ is trivialized by $(\s
S,\mathcal L),$ and is AR-isomorphic to $M,$ so we may
assume that $M$ is adic.

The functor $u_n^*$ induces an equivalence on the
cores with respect to the standard $t$-structures:
$$
(u_n^*)^{\heartsuit}:D^b_{\s S_V,\mathcal L_V}(\mathcal
X_V,\Lambda_n)^{\heartsuit}\to D^b_{\s S,\mathcal L}(\mathcal
X,\Lambda_n)^{\heartsuit}.
$$
Let $M_V=\{M_{n,V},\rho_{n,V}:M_{n,V}\to M_{n-1,V}\}$ be
the unique (up to isomorphism) extension of $M$ to
$\mathcal X_V,$ where $M_{n,V}$ (resp.\ $\rho_{n,V}$) is an
object (resp.\ a morphism) in $D^b_{\s
S_V,\mathcal L_V}(\mathcal X_V,\Lambda_n)^{\heartsuit}$. The induced
morphism $\overline{\rho}_{n,V}:M_{n,V}\otimes\Lambda_{n-1}
\to M_{n-1,V}$ is an isomorphism because it is sent to the
isomorphism $\overline{\rho}_n:M_n\otimes\Lambda_{n-1}\to
M_{n-1}$ via the equivalence $(u_{n-1}^*)^{\heartsuit}.$ This shows that
$M_V$ is an adic system of sheaves on $\mathcal X_V,$ each
level being trivialized by $(\s S_V,\mathcal L_V),$ and it
gives an object in $\s D^b_{\s S_V,\mathcal L_V}(\s
A(\mathcal X_V))$ whose image in $D^b_{\s S_V,\mathcal
L_V}(\mathcal X_V,\Lambda)$ is sent to $K$ under $u^*:$
$$
\xymatrix@C=.3cm @R=.1cm{
\s D^b_{\s S_V,\mathcal L_V}(\s A(\mathcal X_V))
\ar[rrrrr]^-{Q_V} \ar[dddd]_-{u'^*} &&&&& D^b_{\s
S_V,\mathcal L_V}(\mathcal X_V,\Lambda) \ar[dddd]^-{u^*} \\
& M_V \ar@{|->}[rrr] \ar@{|->}[dd] &&& [M_V] \ar@{|..>}[dd]
& \\ &&&&&& \\
& M \ar@{|->}[rrr] &&& K & \\
\s D^b_{\s S,\mathcal L}(\s A(\mathcal X)) \ar[rrrrr]_-Q
&&&&& D^b_{\s S,\mathcal L}(\mathcal X,\Lambda),}
$$
the functors $Q$ and $Q_V$ in the diagram being the localization functors.
This shows that $u^*$ (and similarly, $i^*$) is essentially
surjective.
\end{proof}

\begin{subremark}\label{referee21}
Note that, in contrast to the case of schemes, $a_*\s Ext^q_{\mathcal X_S}(F,G)$ in the proof is in general an \textit{unbounded} complex, so we need to explain why one can shrink $S$ such that all cohomology sheaves are lcc (for all $q$, too). In the proof of (\cite{Sun}, Th. 3.9), that stratifiable complexes are stable under the six operations, we actually proved more, namely, given a pair $(\s S,\mathcal L)$ on $\mathcal X$ and a morphism $f:\mathcal X\to \mathcal Y$ of $S$-algebraic stacks, there exists a pair $(\s S',\mathcal L')$ on $\mathcal Y$ such that $f_*$ takes $D^+_{\s S,\mathcal L}(\mathcal X,\Lambda)$ into $D^+_{\s S',\mathcal L'}(\mathcal Y,\Lambda);$ similar results hold for the other operations, as well as for $\Lambda_n$-coefficients. So in our case, since the $\Lambda_0$-sheaves of the form $j_!L$ are finite in number, we see that the sheaves $R^pa_*\s Ext^q_{\mathcal X_S}(F,G)$ (for all $p, q\in\bb Z$ and $F,G$ of the form $j_!L$) are trivialized by a pair $(\s S',\mathcal L')$ on $S$, and consequently we may replace $S$ by an affine open subscheme in an open stratum in $\s S'.$
\end{subremark}

\begin{anitem}\label{abelian-rank}
We will show that if the $\bb C$-algebraic stack $\mathcal
X$ has affine stabilizers (\ref{affine-stab}), then
$\mathcal X_s$ obtained as above also has affine stabilizers.

It is not difficult to see that the formation of the inertia
stack is compatible with base change in the following sense.
For a 2-Cartesian diagram
$$
\xymatrix@C=1.2cm @R=.7cm{
\mathcal X' \ar[r]^-{g'} \ar[d]_-{f'} & \mathcal X \ar[d]^-f \\
\mathcal Y' \ar[r]_-g & \mathcal Y}
$$
of algebraic stacks over any base $S$ (not necessarily locally of
finite type), let $\mathcal I_f$ and $\mathcal I_{f'}$ be the
relative inertia stacks for $f$ and $f'$ respectively, then
$\mathcal I_{f'}\simeq\mathcal I_f\times_{\mathcal Y}\mathcal
Y'\simeq\mathcal I_f\times_{\mathcal X}\mathcal X'.$ Also, if
$i:\mathcal V\to\mathcal X$ is an immersion, then the restriction
of $\mathcal I_{\mathcal X/S}$ to $\mathcal V$ is $\mathcal
I_{\mathcal V/S},$ i.e.\ $\mathcal I_{\mathcal V/S}\simeq
\mathcal I_{\mathcal X/S}\times_{\mathcal X}\mathcal V.$

Let $P:X_V\to\mathcal X_V$ be a
presentation, and let the following squares be 2-Cartesian:
$$
\xymatrix@C=.52cm @R=.42cm{
& I \ar[rr] \ar'[d][dd] \ar[ld] && I_V \ar'[d][dd] \ar[ld]
&& I_s \ar[dd] \ar[ld] \ar[ll] \\
\mathcal I \ar[rr] \ar[dd] && \mathcal I_V \ar[dd] &&
\mathcal I_s \ar[ll] \ar[dd] & \\
& X \ar[ld] \ar'[r][rr] && X_V \ar[ld] && X_s \ar'[l][ll]
\ar[ld] \\
\mathcal X \ar[rr] && \mathcal X_V && \mathcal X_s. \ar[ll]}
$$
Since any $\bb C$- or $s$-point of $\mathcal X_V$ can be
lifted to $X_V,$ we may replace $\mathcal X_V$ by $X_V$. Since $X_V$ is flat over $\text{Spec }V,$ and the generic point $\eta\in\text{Spec }V$ is an open subset, by (EGA IV, 2.3.10) we see that the generic fiber $X_{\eta}$ is dense in $X_V.$ This is also true with $X_V$ replaced by any stratum in $P^*\s S_V$ (by our assumption that any stratum in $\s S_S$ is $S$-smooth, a fortiori, $S$-flat).

We may assume that $I_V$ is flat over $X_V,$ by stratifying $X_V.$
Replacing $X_V$ by its maximal reduced
subschemes if necessary, we may assume that
$X_V$ is integral. Let $x$ be its generic point, which is
a field of characteristic 0. Therefore $I_{V,x}$ is smooth
over $x,$ hence $I_V$ is smooth over a dense open subset of
$X_V.$ By noetherian induction, we may assume that $I_V$ is smooth over $X_V.$ Then we apply the lower semi-continuity of abelian ranks for smooth group schemes (\cite{SGA3}, X, 8.7), to deduce that all fibers of $I_V\to X_V$ are affine (note that all fibers of $I_{\eta}\to X_{\eta}$ are affine).

\end{anitem}

\begin{anitem}\label{finidiag}
Also, if $f:\mathcal X\to\mathcal Y$ is a proper morphism of finite diagonal between $\bb C$-algebraic stacks, then one can choose $S$ and $V$ such that $f$ extends to a proper morphism of finite diagonal $f_V:\mathcal X_V\to\mathcal Y_V$ between $V$-algebraic stacks. Clearly one has a proper extension $f_V.$ The base change of the diagonal morphism
\[
\Delta_{f_V}:\mathcal X_V\to\mathcal X_V\times_{\mathcal Y_V}\mathcal X_V
\]
to the geometric generic point $\text{Spec }\bb C\to\text{Spec }V$ is the diagonal morphism of $f$
\[
\Delta_f:\mathcal{X\to X\times_YX}.
\]
As $f_V$ is separated, $\Delta_{f_V}$ is representable and proper, so it suffices to show that $\Delta_{f_V}$ is quasi-finite (EGA IV, 8.11.1), which is equivalent to 
$\mathcal I_{f_V}\to\mathcal X_V$ being quasi-finite. As in (\ref{abelian-rank}), we may replace $\mathcal X_V$ by a presentation $X_V$ (and replace $\mathcal I_{f_V}$ by $I_{f_V}:=\mathcal I_{f_V}\times_{\mathcal X_V}X_V$ as well), and stratify $X_V$ to assume that $I_{f_V}$ is a flat group scheme over $X_V$ (assumed integral). Now $I_{f_{\eta}}$ is finite over $X_{\eta}$ (as $\eta\hookrightarrow\bb C$ is faithfully flat), and $X_{\eta}$ is dense in $X_V$ as before, by (\cite{SGA3}, VI$_\text{B}$, Cor.\ 4.3) we see that $I_{f_V}$ is quasi-finite over $X_V.$

In particular, the special fiber $f_s:\mathcal X_s\to\mathcal Y_s$ is also proper and of finite diagonal.

\end{anitem}

\subsection{The proof}

Let $\mathcal X$ be a $\bb C$-algebraic stack, with analytification $\fk X.$
Let $\Omega$ be a field of characteristic 0; the
examples that we have in mind are $\Omega=\bb Q,\ \bb C,\ E_{\lambda}$ or $\overline{\bb Q}_{\ell}.$

\begin{anitem}\label{analytic-perverse}
Following the idea of \cite{LO3}, one can define \textit{$\Omega$-perverse
sheaves} on $\fk X_{\text{lis-an}}$ as
follows. Let $P:X\to\mathcal X$ be a presentation of relative
dimension $d,$ and let $P^{\text{an}}:X(\bb C)\to\fk X$ be
its analytification. Let $p=p_{1/2}$ be the middle perversity on
$X(\bb C).$ Define $\leftexp{p}{\s D}_c^{\le0}(\fk
X,\Omega)$ (resp.\ $\leftexp{p}{\s D}_c^{\ge0}(\fk
X,\Omega)$) to be the full subcategory of objects $K\in\s D_c(\fk
X,\Omega)$ such that $P^{\text{an},*}K[d]$ is in
$\leftexp{p}{\s D}_c^{\le0}(X(\bb C),\Omega)$
(resp.\ $\leftexp{p}{\s D}_c^{\ge0}(X(\bb C),\Omega)$).
As in (\cite{LO3}, 4.1, 4.2), one can show that
these subcategories do not depend on the choice of the
presentation $P,$ and they define a $t$-structure, called
the (\textit{middle}) \textit{perverse $t$-structure} on $\fk X.$
\end{anitem}

\begin{anitem}\label{geometric-origin}
Following (\cite{BBD}, 6.2.4), one can define complexes of sheaves
\textit{of geometric origin} as follows. Let $\s F$ be a
$\Omega$-perverse sheaf on $\fk X_{\text{lis-an}}$ (resp.\ a
$\overline{\bb Q}_{\ell}$-perverse sheaf on $\mathcal
X_{\text{lis-\'et}}$). We say that $\s F$ is
\textit{semi-simple of geometric origin} if it is a
semi-simple perverse sheaf, and every irreducible constituent belongs to
the smallest family of simple perverse sheaves on complex
analytic stacks (resp.\ lisse-\'etale sites of
$\bb C$-algebraic stacks) that

(a) contains the constant sheaf $\underline{\Omega}$
over a point, and is stable under the following
operations:

(b) taking the constituents of $\leftexp{p}{\s H}^i
T,$ for $T=f_*,f_!,f^*,f^!,R\s Hom(-,-)$ and
$-\otimes-,$ where $f$ is an arbitrary
\textit{algebraic} morphism between stacks.

A complex $K\in\s D_c^b(\fk X,\Omega)$ (resp.\
$K\in D_c^b(\mathcal X,\overline{\bb Q}_{\ell})$) is said
to be \textit{semi-simple of geometric origin} if it is
isomorphic to the direct sum of the $(\leftexp{p}{\s
H}^iK)[-i]$'s, and each $\leftexp{p}{\s H}^iK$ is
semi-simple of geometric origin. Notice that this property
is not local for the smooth topology, as the example in
(\cite{Sun2}, Section 1) shows.

One can replace the constant sheaf $E_{\lambda}$ by its ring
of integers $\s O_{\lambda},$ and deduce that every complex $K\in\s
D^b_c(\fk X,\overline{\bb Q}_{\ell})$ that is semi-simple of
geometric origin has an integral structure, hence belongs to the
essential image of $D^b_c(\fk X,\overline{\bb Q}_{\ell})\to\s
D^b_c(\fk X,\overline{\bb Q}_{\ell}).$ Therefore, we can apply
(\ref{R-rat}).
\end{anitem}

\begin{sublemma}\label{6.2.6}\emph{(stack version of
(\cite{BBD}, 6.2.6))}
Let $\s F$ be a simple $\overline{\bb Q}_{\ell}$-perverse sheaf
of geometric origin on $\mathcal X.$ For $A\subset\bb C$ large enough,
the equivalence (\ref{6.1.9})
$$
D^b_{\s S,\mathcal L}(\mathcal X,\overline{\bb
Q}_{\ell}) \longleftrightarrow D^b_{\s S_s,\mathcal
L_s}(\mathcal X_s,\overline{\bb Q}_{\ell})
$$
takes $\s F$ to a simple perverse sheaf $\s
F_s$ on $\mathcal X_s,$ such that $(\mathcal X_s,\s
F_s)$ is deduced by base extension from a pair $(\mathcal
X_0,\s F_0)$ defined over a finite field $\bb
F_q,$ and $\s F_0$ is $\iota$-pure.
\end{sublemma}

\begin{proof}
Being of geometric origin, $\s F_s$ is obtained by base extension
from some simple perverse sheaf $\s F_0$ on $\mathcal X_0,$
which is $\iota$-mixed
by Lafforgue's result. Then apply (\cite{Sun2}, 3.4).
\end{proof}

Finally, we are ready to prove the stack version of the
decomposition theorem over $\bb C.$

\begin{subtheorem}\label{6.2.5}\emph{(stack version of
(\cite{BBD}, 6.2.5))}
Let $f:\mathcal X\to\mathcal Y$ be a proper morphism of
finite diagonal between $\bb C$-algebraic stacks with
affine stabilizers. If $K\in\s D_c^b(\fk X,\Omega)$ is
semi-simple of geometric origin, then $f^{\emph{an}}_*K$
is also bounded, and is
semi-simple of geometric origin on $\fk Y.$
\end{subtheorem}

\begin{proof}
We can replace $\s D_c^b(\fk X,\Omega)$ by $\s D^b_c(\fk
X,\overline{\bb Q}_{\ell}),$ then by $D_c^b(\fk X,\overline{\bb
Q}_{\ell})$ (using (\ref{agree} ii, \ref{R-rat})), and finally
by $D_c^b(\mathcal X,\overline{\bb Q}_{\ell})$ (using
(\ref{compar-cohom})).

From (\cite{Ols1}, 5.17) we know that there is a canonical
isomorphism $f_!\simeq f_*$ on $D_c^-(\mathcal
X,\overline{\bb Q}_{\ell}).$ For $K\in D_c^b,$ we
have $f_!K\in D_c^-$ and $f_*K\in D_c^+,$ hence $f_*K\in
D_c^b.$

\begin{sublemma}
We can reduce to the case where $K$ is a simple perverse
sheaf $\s F.$
\end{sublemma}

\begin{proof}
First, we show that the statement
for simple perverse sheaves of geometric origin implies
the statement for semi-simple perverse sheaves of
geometric origin. This is clear:
$$
f_*(\bigoplus_i\s F_i)=\bigoplus_if_*\s
F_i=\bigoplus_i\bigoplus_j\leftexp{p}{\s H}^j(f_*
\s F_i)[-j]=\bigoplus_j\leftexp{p}{\s H}^j
(f_*(\bigoplus_i\s F_i))[-j].
$$

Then we show that the statement for semi-simple perverse
sheaves implies the general statement. If $K$ is
semi-simple of geometric origin, we have
$$
f_*K=\bigoplus_if_*\ \leftexp{p}{\s H}^i(K)[-i]=
\bigoplus_i\bigoplus_j\leftexp{p}{\s H}^jf_*\
\leftexp{p}{\s H}^i(K)[-i-j].
$$
Taking $\leftexp{p}{\s H}^n$ on both sides, we get
$$
\leftexp{p}{\s H}^n(f_*K)=\bigoplus_{i+j=n}
\leftexp{p}{\s H}^jf_*\ \leftexp{p}{\s H}^i(K),
$$
therefore $f_*K=\bigoplus_n\leftexp{p}{\s H}^n(f_*K)[-n]$ and
each summand is semi-simple of geometric origin.
\end{proof}

Now assume that $K$ is a simple perverse sheaf $\s F,$ which is
stratifiable (\cite{Sun}, 3.4 v). By (\ref{6.2.6}), $\s F$
corresponds to a simple perverse sheaf $\s F_s$ which is induced
from an $\iota$-pure perverse sheaf $\s F_0$ by base change. By
(\ref{Tbc}), the formation of $f_*$ over $\bb C$ is the same as the
formation of $f_{s,*}$ over $\bb F$ and of $f_{0,*}$ over a finite
field. By (\ref{abelian-rank}, \ref{finidiag}) and (\cite{Sun2}, 3.9 iii), $f_{0,*}\s F_0$ is also
$\iota$-pure. By (\cite{Sun2}, 3.11, 3.12), we have
$$
f_{s,*}\s F_s\simeq\bigoplus_{i\in\bb Z}
\leftexp{p}{\s H}^i(f_{s,*}\s F_s)[-i],
$$
and each $\leftexp{p}{\s H}^i(f_{s,*}\s F_s)$ is
semi-simple of geometric origin. Therefore $f_*\s F$
(and hence $f^{\text{an}}_*\s F^{\text{an}}$) is
semi-simple of geometric origin.
\end{proof}

We hope to work out Saito's theory of mixed Hodge modules
for complex analytic stacks in the future, which may lead to an alternative proof of the Decomposition theorem.

\vskip1truecm

\begin{flushleft}
Shenghao SUN \\
Yau Mathematical Sciences Center \\
Tsinghua Univ. \\
Beijing 100084, P. R. China
\end{flushleft}

\end{document}